\documentclass{amsart}

\usepackage{amssymb, amsmath}
\usepackage{mathrsfs}
\usepackage{amscd}
\usepackage{verbatim}
\usepackage{enumerate}
\usepackage{url}

\usepackage{color}

\usepackage[draft=false, colorlinks=true]{hyperref}

\theoremstyle{plain}
\newtheorem{theorem}{Theorem}[section]
\theoremstyle{remark}
\newtheorem{remark}[theorem]{Remark}

\newtheorem{example}[theorem]{Example}

\theoremstyle{plain}

\newtheorem{lemma}[theorem]{Lemma}
\newtheorem{proposition}[theorem]{Proposition}
\newtheorem{definition}[theorem]{Definition}

\numberwithin{equation}{section}
\def\N{{\mathbb N}}
\def\Z{{\mathbb Z}}

\def\R{{\mathbb R}}
\def\C{{\mathbb C}}

\newcommand{\F}{{\mathcal F}}

\renewcommand{\O}{\Omega}

\newcommand{\calL}{{\mathcal L}}

\newcommand{\one}{{{\bf 1}}}

\newcommand{\lb}{\langle}
\newcommand{\rb}{\rangle}

\newcommand{\limn}{\lim_{n\to\infty}}

\newcommand{\wh}{\widehat}
\newcommand{\supp}{\text{\rm supp\,}}

\newcommand{\ud}{\, d}
\newcommand{\Tr}{{\rm tr}}
\newcommand{\ext}{{\rm ext}}
\newcommand{\sign}{\text{sign}}
\newcommand {\Schw}{\mathcal{S}}
\newcommand {\Distr}{\mathcal{D}}
\newcommand{\Do}{\mathrm D}
\newcommand{\wt}{\widetilde}
\newcommand{\wtE}{\wt{\mathcal{E}}_+^m}
\newcommand{\wtEr}{\wt{E}_+^m}

\begin{document}

\author{Nick Lindemulder}
\address{Delft Institute of Applied Mathematics\\
Delft University of Technology \\ P.O. Box 5031\\ 2600 GA Delft\\The
Netherlands} \email{N.Lindemulder@tudelft.nl}
\author{Martin Meyries}
\address{Institut f\"ur Mathematik\\ Martin-Luther-Universit\"at Halle-Wittenberg\\ 06099 Halle
(Saale)\\ Germany} \email{martin.meyries@dataqube.de}
\author{Mark Veraar}
\address{Delft Institute of Applied Mathematics\\
Delft University of Technology \\ P.O. Box 5031\\ 2600 GA Delft\\The
Netherlands} \email{M.C.Veraar@tudelft.nl}

\thanks{The first and third author
are supported by the VIDI subsidy 639.032.427 of the Netherlands Organisation for Scientific Research (NWO)}

\date\today

\title[Complex interpolation with Dirichlet boundary conditions]{Complex interpolation with Dirichlet boundary conditions on the half line}

\begin{abstract}
We prove results on complex interpolation of vector-valued Sobolev spaces over the half-line with Dirichlet boundary condition. Motivated by applications in evolution equations, the results are presented for Banach space-valued Sobolev spaces with a power weight. The proof is based on recent results on pointwise multipliers in Bessel potential spaces, for which we present a new and simpler proof as well. We apply the results to characterize the fractional domain spaces of the first derivative operator on the half line.
\end{abstract}



\maketitle

\section{Introduction}
The main result of the present paper is the following. Let $W^{1,p}_0(\R_+;X)$ be the first order Sobolev space over the half line with values in a UMD Banach space $X$ vanishing at $t=0$, where $p\in (1, \infty)$. Then for complex interpolation we have
$$[L^p(\R_+;X),W^{1,p}_0(\R_+;X)]_{\theta} = H^{\theta, p}_0(\R_+; X), \qquad \theta \in (0,1), \quad \theta \neq 1/p,$$
see Theorems \ref{thm:compl_int_bd-cond_half-space} and \eqref{eq:W0H0}. Here $H^{\theta, p}_0$ denotes the fractional order Bessel potential space with vanishing trace for $\theta > 1/p$, and $H^{\theta, p}_0 = H^{\theta, p}$ for $\theta < 1/p$. In more generality, we consider spaces with Muckenhoupt power weights $w_\gamma(t) = t^\gamma$, where the critical value $1/p$ is shifted accordingly.

In the scalar-valued case $X=\C$, the result is well-known and due to Seeley \cite{Se}. The vector-valued result was already used several times in the literature without proof. Seeley also considers the case $\theta=1/p$, which we ignore throughout for simplicity, and the case of domains $\Omega\subseteq \R^d$. The corresponding result for real interpolation is due to Grisvard \cite{Grisvard} and more elementary to prove.

At the heart of complex interpolation theory with boundary conditions is the pointwise multiplier property of the characteristic function of the half-space $\one_{\R_+}$ on $H^{\theta,p}(\R; X)$ for $0<\theta<1/p$. It is due to Shamir \cite{Shamir} and Strichartz \cite{Strichartz} in the scalar-valued case. In \cite{MeyVerpoint} by the second and third author, a general theory of pointwise multiplication of weighted vector-valued functions was developed. As a main application the multiplier result was extended to the vector-valued and weighted setting. An alternative approach to this was found by the first author in \cite{Lindiff16} and is based on a new equivalent norm for vector-valued Bessel potential spaces. In Section \ref{sec:pointwise} we present a new and simpler proof of the multiplier property of $\one_{\R_+}$, which is based on the representation of fractional powers of the negative Laplacian as a singular integral and the Hardy-Hilbert inequality.

For future reference and as it is only a minimal extra effort, we will formulate and prove some elementary assertions for the half space $\R_+^d$ for $d\geq 1$ or even domains, and general $A_p$ weights $w$. In order to make the presentation as self-contained as possible, we further fully avoid the use of Triebel--Lizorkin spaces and Besov spaces, but we point out where they could be used.  We will only use the UMD property of $X$ through standard applications of the Mihlin multiplier theorem. Several results will be presented in such a way that the UMD property is not used. A detailed explanation of the theory of UMD spaces and their connection to harmonic analysis can be found in the monograph \cite{HNVW1}. In their reflexive range, all standard function spaces are UMD spaces.

The complex interpolation result has applications in the theory of evolution equations, as it yields a characterization of the fractional power domains of the time derivative $\Do((d/dt)^\theta)$ and $\Do((-d/dt)^\theta)$ on $\R_+$. Here the half line usually stands for the time variable and $X$ is a suitable function space for the space variable.  For instance such spaces can be used in the theory of Volterra equations (see \cite{Pruvolt,Zac03, Zac05}), in evolution equations with form methods (see \cite{dier2016non,fackler2016non}), in stochastic evolution equations (see \cite{NVW12a}).

In order to deal with rough initial values it is useful to consider a power weights $w_\gamma(t)=t^{\gamma}$ in the time variable. Examples of papers in evolution equation where such weights are used include \cite{Angenent90, ClSi01, KPW, lindemulder2017maximal, MeySchn, MeySch, MeyVerTr, PrSi, PrSiZa}. The monographs \cite{Am,Lun,PrSibook} are an excellent source for applications of weighted spaces to evolution equations. In order to make our results available to this part of the literature as well, we present our interpolation results for weighted spaces.
For the application to evolution equations it suffices to consider interpolation of vector-valued Sobolev spaces over $\R_+$ with Dirichlet boundary conditions and therefore we focus on this particular case. In a future paper we extend the results of \cite{Grisvard} and \cite{Se} to weighted function spaces on more general domains $\Omega\subseteq \R^d$, in the scalar valued situation, where one of the advantages is that Bessel potential spaces have a simple square function characterization.

\subsubsection*{Overview}
\begin{itemize}
\item In Section \ref{sec:prel} we discuss some preliminaries from harmonic analysis.
\item In Section \ref{sec:weighted} we introduce the weighted Sobolev spaces and Bessel potential spaces.
\item In Section \ref{sec:pointwise} we present an elementary proof of the pointwise multiplier theorem.
\item In Section \ref{sec:interpo} we present some results on interpolation theory without boundary conditions.
\item In Section \ref{sec:intbdr} we present the main results on interpolation theory with boundary conditions and applications to fractional powers.
\end{itemize}

\subsubsection*{Notation}

$\R^d_+ = (0,\infty)\times\R^{d-1}$ denotes the half space. We write $x = (x_1, \tilde{x})\in \R^{d}$ with $x_1\in \R$ and $\tilde{x}\in \R^{d-1}$ and define the weight $w_{\gamma}$ by $w_{\gamma}(x_1, \tilde{x}) = |x_1|^{\gamma}$. Sometimes it will be convenient to also write $(t,x)\in \R^d$ with $t\in \R$ and $x\in \R^{d-1}$. The operator $\F$ denotes the Fourier transform. We write $A \lesssim_p B$ whenever $A \leq C_p B$ where $C_p$ is a constant which depends on the parameter $p$. Similarly, we write $A\eqsim_p B$ if $A\lesssim_p B$ and $B\lesssim_p A$.

\subsection*{Acknowledgements}
We thank the anonymous referees for their helpful comments.

\section{Preliminaries\label{sec:prel}}

\subsection{Weights}

A locally integrable function $w:\R^d\to (0,\infty)$ will be called a {\em weight function}. Given a weight function $w$ and a Banach space $X$ we define $L^p(\R^d,w;X)$ as the space of all strongly measurable $f:\R^d\to X$ for which
\[\|f\|_{L^p(\R^d,w;X)} := \Big(\int \|f(x)\|^p w(x) \ud x\Big)^{\frac1p}\]
is finite. Here we identify functions which are a.e.\ equal.

Although we will be mainly interested in a special class of weights, it will be natural to formulate some of the result for the class of Muckenhoupt $A_p$-weights. For $p\in (1, \infty)$, we say that $w\in A_p$ if
\[[w]_{A_p} = \sup_{Q} \frac{1}{|Q|} \int_Q w(x) \ud x \cdot \Big(\frac{1}{|Q|} \int_Q w(x)^{-\frac{1}{p-1}}\ud x\Big)^{p-1}<\infty.\]
Here the supremum is taken over all cubes $Q\subseteq \R^d$ with sides parallel to the coordinate axes. For $p\in (1, \infty)$ and a weight $w:\R^d\to (0,\infty)$ one has $w\in A_p$ if and only the Hardy--Littlewood maximal function is bounded on $L^p(\R^d,w)$.
We refer the reader to \cite[Chapter 9]{GrafakosM} for standard properties of $A_p$-weights. For a fixed $p$ and a weight $w\in A_p$, the weight $w' = w^{-1/(p-1)}\in A_{p'}$ is the $p$-dual weight. By H\"older's inequality one checks that
\begin{equation}\label{eq:fgdualityweight}
\int |f(x)|  |g(x)| \ud x\leq  \|f\|_{L^p(\R^d,w)} \|g\|_{L^{p'}(\R^d,w')}
\end{equation}
for $f\in L^p(\R^d,w)$ and $g\in L^{p'}(\R^d,w')$. Using this, for each $w\in A_p$ one can check that $L^p(\R^d,w;X)\subseteq L^1_{\rm loc}(\R^d;X)$.

The following will be our main example.
\begin{example}
Let
\[w_{\gamma}(x_1, \tilde{x}) = |x_1|^{\gamma}, \ \ x_1\in \R, \tilde{x}\in \R^{d-1}.\]
As in \cite[Example 9.1.7]{GrafakosM}) one sees that $w_{\gamma}\in A_p$ if and only if $\gamma\in (-1, p-1)$.
\end{example}

\begin{lemma}\label{lem:convolution}
Let $p\in (1, \infty)$ and $w\in A_p$. Assume $\phi\in L^1(\R^d)$ and $\int \phi \ud x = 1$. Let $\phi_n(x) = n^d\phi(nx)$. Assume $\phi$ satisfies any of the following conditions:
\begin{enumerate}
\item $\phi$ is bounded and compactly supported
\item There exists a radially decreasing function $\psi\in L^1(\R^d)$ such that $|\phi|\leq \psi$ a.e.
\end{enumerate}
Then for all $f\in L^p(\R^d;X)$, $\phi_n*f\to f$ in $L^p(\R^d,w;X)$ as $n\to \infty$. Moreover, there is a constant $C$ only depending on $\phi$ such that $\|\phi_n *f\| \leq C Mf$ almost everywhere.
\end{lemma}
\begin{proof}
For convenience of the reader we include a short proof. By \cite[Theorem 2.40 and Corollary 2.41]{HNVW1} $\phi_n*f\to f$ almost everywhere and $\|\phi_n*f\|\leq \|\psi\|_{L^1(\R^d)} Mf$ almost everywhere, where $M$ denotes the Hardy--Littlewood maximal function. Therefore, the result follows from the dominated convergence theorem.
\end{proof}

\subsection{Fourier multipliers and UMD spaces}

Let $\Schw(\R^d;X)$ be the space of $X$-valued Schwartz functions and let $\Schw'(\R^d;X) = \calL(\Schw(\R^d), X)$ be the space of $X$-valued tempered distributions. For $m\in L^\infty(\R^d)$ let $T_m:\Schw(\R^d;X)\to \Schw'(\R^d;X)$ be the Fourier multiplier operator defined by
\[T_m f = \F^{-1} (m \wh{f}).\]
There are many known conditions under which $T_m$ is a bounded linear operator on $L^p(\R^d;X)$.
In the scalar-valued the set of all Fourier multiplier symbols on $L^{2}(\R^{d})$ for instance coincides with $L^{\infty}(\R^{d})$.
In the case $p\in (1, \infty)\setminus\{2\}$ a large set of multipliers for which $T_m$ is bounded is given by Mihlin's multiplier theorem. In the vector-valued case difficulties arise and geometric conditions on $X$ are needed already if $d = 1$ and $m(\xi) = \sign(\xi)$; in fact, in \cite{Bour,Bu83} it was shown that in this specific case the boundedness of $T_{m}$ on $L^{p}(\R;X)$ characterizes the UMD property of $X$.
Since the work of \cite{Bour,Bu83,McCon84} it is well-known that the right class of Banach spaces for vector-valued harmonic analysis is the class of UMD Banach spaces, as many of the classical results in harmonic analysis, such as the classical Mihlin multiplier theorem, have been extended to this setting.  We refer to \cite{Burk01, HNVW1, Rub} for details on UMD spaces and Fourier multiplier theorems.

All UMD spaces are reflexive. Conversely, all spaces in the reflexive range of the classical function spaces have UMD: e.g.: $L^p$, Bessel potential spaces, Besov spaces, Triebel-Lizorkin spaces, Orlicz spaces.

The following result is a weighted version of the Mihlin multiplier theorem which can be found in \cite[Proposition 3.1]{MeyVerpoint} and is a simple consequence of \cite{HH14}.
\begin{proposition}\label{prop:Mihlin}
Let $X$ be a UMD space, $p\in (1, \infty)$ and $w\in A_p$. Assume that $m\in C^{d+2}(\R^d\setminus\{0\})$ satisfies
\[C_{m} := \sup_{|\alpha| \leq d+2}\sup_{\xi \neq 0} |\xi|^{|\alpha|} |\partial^{\alpha} m(\xi)|<\infty.\]
Then $T_m$ is bounded on $L^p(\R^d,w;X)$ and has an operator norm that only depends $C_{m},d,p,X,[w]_{A_{p}}$.
\end{proposition}

\section{Weighted function spaces\label{sec:weighted}}

In this section we present several results on weighted function spaces, which do not require the UMD property of the underlying Banach space (except in Proposition \ref{prop:W=H}).

\subsection{Definitions and basic properties}

For an open set $\Omega\subseteq \R^d$ let $\mathcal{D}(\Omega)$ denote the space compactly supported smooth functions on $\Omega$ equipped with its usual inductive limit topology.
For a Banach space $X$, let $\Distr'(\Omega;X) = \calL(\mathcal{D}(\Omega),X)$ be the space of $X$-valued distributions. For a distribution $u\in \Distr'(\Omega;X)$ and an open subset $\O_0\subseteq \O$, we define the restriction $u|_{\O_0}\in \Distr'(\Omega_0;X)$ as $u|_{\O_0}(f) = u(f)$ for $f\in \mathcal{D}(\Omega_{0})$.

For $p \in (1,\infty)$ and $w\in A_p$ let $W^{k,p}(\Omega,w;X)\subseteq \Distr'(\Omega;X)$ be the {\em Sobolev space} of all $f\in L^p(\Omega,w;X)$ with $\partial^\alpha f \in L^p(\Omega,w;X)$ for all $|\alpha|\leq k$ and set
\begin{align*}
\|f\|_{W^{k,p}(\Omega,w;X)} &= \sum_{|\alpha|\leq k} \|\partial^{\alpha} f\|_{L^p(\Omega,w;X)}, \\ [f]_{W^{k,p}(\Omega,w;X)} &= \sum_{|\alpha| = k} \|\partial^{\alpha} f\|_{L^p(\Omega,w;X)}.
\end{align*}
Here for $\alpha\in \N^d$, $\partial^{\alpha} = \partial_1^{\alpha_1} \ldots \partial_d^{\alpha_d}$.

Let $\mathcal{J}_{s}$ denote the Bessel potential operator of order $s\in \R$ defined by \[\mathcal{J}_{s} f = (1-\Delta)^{s/2} f := \F^{-1} (1+|\cdot|^2)^{s/2} \wh{f},\]
where $\wh{f}$ denotes the Fourier transform of $f$ and $\Delta = \sum_{j=1}^d \partial_{j}^2$. For $p\in (1, \infty)$, $s\in \R$ and $w\in A_p$ let $H^{s,p}(\R^d,w;X)\subseteq \Schw'(\R^d;X)$ denote the {\em Bessel potential space} of all $f\in  \Schw'(\R^d;X)$ for which $\mathcal{J}_{s} f\in L^p(\R^d,w;X)$ and set
\[\|f\|_{H^{s,p}(\R^d,w;X)} = \|\mathcal{J}_{s} f\|_{L^p(\R^d,w;X)}.\]

In the following lemma we collect some properties of the operators $\mathcal{J}_s$.
\begin{lemma}\label{lem:propGs}
Fix $s>0$. There exists a function $G_s:\R^d\to [0,\infty)$ such that $G_s\in L^1(\R^d)$ and $\mathcal{J}_{-s} f = G_s*f$ for all $f\in \Schw'(\R^d;X)$. Moreover, $G_s$ has the following properties:
\begin{enumerate}
\item\label{it:expnearinftya} For all $|y|\geq 2$, $G_s(y) \lesssim_{s,d} e^{-\frac{|y|}{2}}$.
\item\label{it:expnearinftyb}
For $|x|\leq 2$,
\[
G_{s}(x) \lesssim_{s,d} \left\{\begin{array}{ll}
|x|^{s-d}, & s\in (0,d), \\
1+\log(\tfrac{2}{|x|}), & s= d, \\
1, & s > d,
\end{array}\right.
\]
\item\label{it:hogered}
for all $s>k\geq 0$ and all $|\alpha|\leq k$, there exists a radially decreasing function $\phi\in L^1(\R^d)$  such that $|\partial^{\alpha} G_{s}|\leq \phi$ pointwise.
\end{enumerate}
In particular, if $d=1$, $p\in (1, \infty)$, $\gamma\in (-1, p-1)$ and $s > \frac{1+\gamma}{p}$, then $G_{s} \in L^{p'}(\R,w_{\gamma}')$.
\end{lemma}
\begin{proof}
The fact that the positive function $G_s\in L^1(\R^d)$ exists,  together with \eqref{it:expnearinftya} and \eqref{it:expnearinftyb}, follows from \cite[Section 6.1.b]{GrafakosM}.

To prove \eqref{it:hogered}, we use the following representation of $G_s$ (see \cite[Section 6.1.b]{GrafakosM}):
\[G_s(x) = C_{s,d} \int_0^\infty e^{-t} e^{-\frac{|x|^2}{4t}} t^{\frac{s-d}{2}} \frac{\ud t}{t}.\]
By induction one sees that $\partial^{\alpha} G_s(x)$ is a linear combination of functions of the form $G_{s - 2j} (x) |x|^{\beta}$ with $|\beta|\leq j\leq k$. Therefore, by \eqref{it:expnearinftyb} for $|x|\leq 2$, $|\partial^{\alpha} G_s(x)|\lesssim_{s,d,\alpha} |x|^{\varepsilon-d}$ for some $\varepsilon\in (0,d)$. On the other hand for $|x|\geq 2$, $|\partial^{\alpha} G_s(x)|\lesssim_{s,d,\alpha} |x|^{\beta} e^{-\frac{|x|}{2}} \lesssim_{d,s,k}  e^{-\frac{|x|}{4}}$. Now the function $\phi(x) = C_1|x|^{\varepsilon-d}$ for $|x|\leq 2$ and $\phi(x) = C_2e^{-\frac{|x|}{4}}$ for certain constants $C_1, C_2>0$. satisfies the required conditions.

To prove the final assertion for $d=1$, note that the blow-up behaviour near $0$ gets worse as $s$ decreases. Therefore, without loss of generality we may assume that $s \in (\frac{1+\gamma}{p},1)$, in which case \eqref{it:expnearinftyb} yields
\[
|G_{s}(x)|^{p'}w'_{\gamma}(x) \lesssim_{s,p,\gamma} |x|^{\frac{(s-1)p-\gamma}{p-1}} =
|x|^{-1+\frac{p}{p-1}(s-\frac{1+\gamma}{p})} \quad \mbox{for $|x| \leq 2$}.
\]
which is integrable. Integrability, for $|x|>2$, is clear from \eqref{it:expnearinftya}.
\end{proof}

The following result is proved in \cite[Proposition 3.2 and 3.7]{MeyVerpoint} by a direct application of Proposition \ref{prop:Mihlin}.
\begin{proposition}\label{prop:W=H}
Let $X$ be a UMD space, $p\in (1, \infty)$, $k\in \N_0$, $w\in A_p$.
Then $H^{k,p}(\R^d,w;X) = W^{k,p}(\R^d,w;X)$ with norm equivalence only depending on $d$, $X$, $p$, $k$ and $[w]_{A_p}$.
\end{proposition}

The UMD property is necessary in Proposition \ref{prop:W=H}
(see \cite[Theorem 5.6.12]{HNVW1}). Sometimes it can be avoided by instead using the following simple embedding result which holds for any Banach space. The sharper version $W^{k,p}(\R^d,w;X)\hookrightarrow H^{s,p}(\R^d,w;X)$ if $s<k$ and $k\in \N_0$. can be obtained from \cite[Propositions 3.11 and 3.12]{MeyVersharp} but is more complicated.

\begin{lemma}\label{lem:compHW}
Let $X$ be a Banach space, $p\in (1, \infty)$, $k\in \N_0$, $s\in (k,\infty)$ and $w\in A_p$. Then the following continuous embeddings hold
\[W^{2k,p}(\R^d,w;X)\hookrightarrow H^{2k,p}(\R^d,w;X), \ \ H^{s,p}(\R^d,w;X)\hookrightarrow W^{k,p}(\R^d,w;X),\]
with embedding constants which only depend on $d, s, k$ and $[w]_{A_p}$.
\end{lemma}
\begin{proof}
The first embedding is immediate from $J_{2k} f = (1-\Delta)^{k} f$ and Leibniz' rule. For the second embedding let $f\in H^{s,p}(\R^d,w;X)$ and write $f_{s} = J_{s} f\in L^p(\R^d,w;X)$. By Lemma \ref{lem:propGs} \eqref{it:hogered} and Lemma \ref{lem:convolution}, for all $|\alpha|\leq k$,
\[\|\partial^{\alpha}f\|_X = \|\partial^{\alpha} G_{s} *f_{s}\|_X\leq \phi*\|f_{s}\|_X\leq C_{\phi} M(\|f_{s}\|_X),\]
where $\phi\in L^1(\R^d)$ is a radially decreasing function depending on $\alpha$, $k$ and $s$.
Therefore, by the boundedness of the Hardy--Littlewood maximal function, we have $\partial^{\alpha}f\in L^p(\R^d,w;X)$ with
\[\|\partial^{\alpha} f\|_{L^p(\R^d,w;X)} \lesssim_{p,[w]_{A_p}}  \|f_{s}\|_{L^p(\R^d,w;X)} = \|f\|_{H^{s,p}(\R^d,w;X)}.\]
Now the result follows by summation over all $\alpha$.
\end{proof}

We proceed with two density results.
\begin{lemma}\label{lem:densityCcH}
Let $X$ be a Banach space, $p\in (1, \infty)$, $s\in \R$ and $w\in A_p$.
Then $\Schw(\R^d;X)\hookrightarrow H^{s,p}(\R^d,w;X) \hookrightarrow \Schw'(\R^d;X)$. Moreover, $C^\infty_c(\R^d)\otimes X$ is dense in $H^{s,p}(\R^d,w;X)$.
\end{lemma}
\begin{proof}
First we prove that $\Schw(\R^d;X)\hookrightarrow H^{s,p}(\R^d,w;X)$. It suffices to prove this in the case $s=0$ by continuity of $\mathcal{J}_{s} = (1-\Delta)^{s/2}$ on $\Schw(\R^d;X)$. In the case $s=0$, the continuity of the embedding follows from
\begin{align*}
\|f\|_{L^p(\R^d,w;X)}& \leq \|(1+|x|^2)^{-n}\|_{L^p(\R^d,w)}  \|(1+|x|^2)^{n} f\|_{L^\infty(\R^d;X)}
\\ & \lesssim_{d,n,p,w} \sum_{|\alpha|\leq 2n} \sup_{x\in \R^d}\|x^{\alpha} f(x)\|
\end{align*}
for $n\in \N$  with $n \geq dp$ (see \cite[Lemma~4.5]{MeyVersharp}).

To prove the density assertion note that $L^p(\R^d,w)\otimes X$ is dense in $L^p(\R^d,w;X)$ and $\Schw(\R^d)$ is dense in $L^p(\R^d,w)$ (see \cite[Exercise 9.4.1]{GrafakosM})
it follows that $\Schw(\R^d)\otimes X$ is dense in $L^p(\R^d,w;X)$. Since $J^{-s}$ leaves $\Schw(\R^d)$ invariant, also $\Schw(\R^d)\otimes X$ is dense in $H^{s,p}(\R^d,w;X)$.
Combining this with $\Schw(\R^d;X)\hookrightarrow H^{s,p}(\R^d,w;X)$ and the fact that $C^\infty_c(\R^d)$ is dense in $\Schw(\R^d)$ (see \cite[Lemma~14.7]{Duistermaat&Kolk_distributies}) we obtain the desired density assertion.

To prove the embedding $H^{s,p}(\R^d,w;X) \hookrightarrow \Schw'(\R^d;X)$ it suffices again to consider $s=0$. In this case from \eqref{eq:fgdualityweight} and $\Schw(\R^d) {\hookrightarrow} L^{p'}(\R^d,w')$ densely, we deduce
\[L^{p}(\R^d,w;X) \hookrightarrow \calL(L^{p'}(\R^d,w'), X)\hookrightarrow \calL(\Schw(\R^d), X) = \Schw'(\R^d;X).\]
\end{proof}

\begin{lemma}\label{lem:densityCcW}
Let $X$ be a Banach space, $p\in (1, \infty)$, $k \in \N$ and $w\in A_p$.
Then $\Schw(\R^d;X)\hookrightarrow W^{k,p}(\R^{d},w;X) \hookrightarrow \Schw'(\R^d;X)$. Moreover, $C^\infty_c(\R^d)\otimes X$ is dense in $W^{k,p}(\R^{d},w;X)$.
\end{lemma}
\begin{proof}
The case $k=0$ follows from Lemma~\ref{lem:densityCcH}
and the case $k\geq 1$ follow by differentiation.

Let $\phi\in C^\infty_c(\R^d)$ be such that $\int_{\R^d} \phi \ud x = 1$ and define $\phi_{n}:=n^{d}\phi(n\,\cdot\,)$ for every $n \in \N$.
Then, by Lemma \ref{lem:convolution} and standard properties of convolutions, $f_n:=\phi_n*f\to f$ in $W^{k,p}(\R^d,w;X)$ as $n \to \infty$ with $\phi_{n} * f \in W^{\infty,p}(\R^{d},w;X) = \bigcap_{l \in \N}W^{l,p}(\R^{d},w;X)$.
In particular, $W^{2k+2,p}(\R^d,w;X)$ is dense in $W^{k,p}(\R^{d},w;X)$.
This yields $H^{k+1,p}(\R^d,w;X) \stackrel{d}{\hookrightarrow} W^{k,p}(\R^{d},w;X)$ by Lemma~\ref{lem:compHW}.
The density of $C^\infty_c(\R^d)\otimes X$ in $W^{k,p}(\R^{d},w;X)$ now follows from Lemma~\ref{lem:densityCcH}.
\end{proof}

\begin{lemma}\label{lem:convolution;Bessel_pot_space}
Let $X$ be a Banach space, $p\in (1, \infty)$, $s\in \R$ and $w\in A_p$.
Assume $\phi\in C^{\infty}_{c}(\R)$ with $\int \phi \ud x = 1$. Let $\phi_n(x) = n^d\phi(nx)$.
Then, for all $f\in H^{s,p}(\R^d,w;X)$,
\[\|\phi_n*f\|_{H^{s,p}(\R^d,w;X)} \lesssim_{s,p,[w],d} \|f\|_{H^{s,p}(\R^d,w;X)}\]
with $\phi_n*f\to f$ in $H^{s,p}(\R^d,w;X)$ as $n\to \infty$ with $\phi_{n}*f \in H^{\infty,p}(\R^d,w;X)= \bigcap_{t \in \R} H^{t,p}(\R^d,w;X)$.
\end{lemma}
\begin{proof}
The first part of the statement follows from Lemma~\ref{lem:convolution} and $\mathcal{J}_{s} (\phi_n*f) = \phi_n*\mathcal{J}_{s} f$.
For the last part, note that $\phi_n*f = \mathcal{J}_{-s}[\phi_n*\mathcal{J}_{s} f] \in H^{\infty,p}(\R^d,w;X)$ by basic properties of convolutions in combination with Lemma~\ref{lem:compHW}.
\end{proof}

The following version of the Hardy inequality will be needed (see \cite[Corolllary 1.4]{MeyVersharp} for a related result). The result can be deduced from \cite[Theorem 1.3 and Proposition 4.3]{MeyVerCh} but for convenience we include an elementary proof.
\begin{lemma}[Hardy inequality with power weights]\label{lem:Hardypp}
Let $\gamma\in (-1, p-1)$ and $s\in (0,1)$. Let $w_{\gamma}(t,x) = |t|^{\gamma}$ for $t\in \R$ and $x\in \R^{d-1}$. Then
$H^{s,p}(\R^{d},w_{\gamma};X) \hookrightarrow L^p(\R^{d},w_{\gamma-s p};X)$.
\end{lemma}
\begin{proof}
It suffices to prove $\|G_s*f\|_{L^p(w_{\gamma-s p};X)} \lesssim_{p,s,d,\gamma} \|f\|_{L^p(w_{\gamma};X)}$, where $G_s$ is as in Lemma \ref{lem:propGs} and $f\in L^p(w_{\gamma};X)$. Since $G_s\geq 0$, by the triangle inequality it suffices to consider the case of scalar functions $f$ with $f\geq 0$.

To prove the result we first apply Minkowski's and Young's inequality in $\R^{d-1}$:
\begin{align*}
\|G_s*f(t,\cdot)\|_{L^p(\R^{d-1})}\leq \int_{\R} \|G_s(t-\tau, \cdot)\|_{L^1(\R^{d-1})} \|f(\tau,\cdot)\|_{L^p(\R^{d-1})} d\tau = g_s*\phi(\tau).
\end{align*}
Here $g_s(t) = \|G_s(t, \cdot)\|_{L^1(\R^{d-1})}$ and $\phi(\tau) = \|f(\tau,\cdot)\|_{L^p(\R^{d-1})}$. Then for $|t|\leq 2$, by Lemma \ref{lem:propGs} \eqref{it:expnearinftya} and \eqref{it:expnearinftyb},
\[g_s(t) \lesssim_{s,d} \int_{\R^{d-1}}  (|t|+|x|)^{s-d} dx = |t|^{s-1} \int_{\R^{d-1}} (1+|x|)^{s-d}  dx =C |t|^{s-1},\]
where we used $s<1$. For $|t|>2$, by Lemma \ref{lem:propGs} \eqref{it:expnearinftyb} and $|(t,x)| \eqsim |t|+|x|$, we find
\[g_s(t) \lesssim_{s,d} e^{- \frac{|t|}{2}} \int_{\R^d} e^{-\frac{|x|}{2}} dx \eqsim_d e^{-\frac{|t|}{2}}.\]

Finally by the weighted version of Young's inequality (see and \cite[Theorem 3.4(3.7)]{Kerman}) in dimension one, we find that
\begin{align*}
\|G_s*f\|_{L^p(\R^d,w_{\gamma-sp})} \leq  \|g_s*\phi\|_{L^p(\R,w_{\gamma-s p})} \leq C \|\phi\|_{L^p(\R,w_{\gamma})} = C\|f\|_{L^p(\R^d,w_{\gamma})},
\end{align*}
where $C = \sup_{t\in \R} |t|^{1-s} g_s(t)<\infty$.
\end{proof}

We end this section with a weighted version of the classical Hardy--Hilbert inequality.
\begin{lemma}[Hardy--Hilbert inequality with power weights]\label{lem:HardyHilbert}
Let $p\in (1, \infty)$ and $\gamma\in (-1,p-1)$. Let $w_{\gamma}(x_1,\tilde{x}) = |x_1|^{\gamma}$ and $k(x,y) = \frac{1}{((|x_1|+|y_1|)^2+ |\tilde{x}-\tilde{y}|^2)^{d/2}}$, where $x = (x_1, \tilde{x})$ and $y = (y_1, \tilde{y})$.
Then the formula
\[I_k h(x) := \int_{\R^d} k(x,y) h(y) \ud y\]
yields a well-defined bounded linear operator $I_k$ on $L^p(\R^d, w_{\gamma})$.
\end{lemma}
\begin{proof}
It suffices to consider $h \geq 0$.
Moreover, by symmetry it is enough to consider $x_1,y_1>0$. Thus we need to show that
\[
\| x \mapsto \int_{\R^d_{+}} k(x,y) h(y) \ud y \|_{L^{p}(\R^{d}_{+},w_{\gamma})} \lesssim_{p, d, \gamma} \|h\|_{L^{p}(\R^{d}_{+},w_{\gamma})},
\qquad h \in L^{p}(\R^{d}_{+},w_{\gamma}), h \geq 0.
\]

{\em Step I.} \emph{The case $d=1$.} Replacing $k$ by
\[k_{\beta}(x,y) =\frac{w_{\gamma}(x)^{1/p} w_{\gamma}(y)^{-1/p}}{(|x|+|y|)} = \frac{|x|^{\beta} |y|^{-\beta}}{|x|+|y|},\]
with $\beta = \gamma/p$, it suffices to consider the unweighted case.

To prove the required result we apply Schur's test in the same way as in \cite[Theorem 5.10.1]{Garling}. Let $s(x) = t(x) = x^{-\frac{1}{pp'}}$. Then since $-1<\beta-\frac{1}{p'}<0$
\begin{align*}
\int_0^\infty s(x)^{p} k_{\beta}(x,y) \ud x = \int_0^\infty \frac{x^{\beta-\frac{1}{p'}} y^{-\beta}}{x+y} \ud x = t(y)^{p} \int_0^\infty \frac{z^{\beta-\frac1p}}{z+1} \ud z = C_{p,\beta} t(y)^p.
\end{align*}
Similarly, since $-1<-\beta-\frac{1}{p}<0$
\begin{align*}
\int_0^\infty t(y)^{p'} k_{\beta}(x,y) \ud y = \int_0^\infty \frac{x^{\beta} y^{-\beta - \frac1p}}{x+y} \ud y = s(x)^{p'} \int_0^\infty \frac{z^{-\beta-\frac1p}}{1+z} \ud z = C_{p,\beta} s(x)^{p'}.
\end{align*}

{\em Step II.} \emph{The general case.}
By Minkowski's inequality we find
\begin{align*}
\|I_k f(x_1, \cdot)\|_{L^p(\R^{d-1})} \leq \int_{0}^\infty \Big(\int_{\R^{d-1}}\Big(\int_{\R^{d-1}} \frac{f(y_1, \tilde{y})}{((x_1+y_1)^2 + |\tilde{x}-\tilde{y}|^2)^{d/2}} \ud \tilde{y}\Big)^p \ud \tilde{x}\Big)^{1/p} \ud y_1.
\end{align*}
Fix $y_1>0$ and let $g_r(\tilde{y}) = f(y_1, r\tilde{y})$. Setting $r = x_1+y_1$ and substituting $u:=\tilde{x}/r$ and $v:=\tilde{y}/r$ we can write
\begin{align*}
\int_{\R^{d-1}}&\Big(\int_{\R^{d-1}} \frac{f(y_1, \tilde{y})}{(|x_1+y_1|^2 + |\tilde{x}-\tilde{y}|^2)^{d/2}} \ud \tilde{y}\Big)^p \ud \tilde{x} \\ & =
r^{-p+d-1} \int_{\R^{d-1}}\Big(\int_{\R^{d-1}} \frac{g_r(v)}{(1 + |u-v|^2)^{d/2}} \ud v\Big)^p \ud u
\\ & \leq r^{-p+d-1} \|g_r\|_{L^p(\R^{d-1})}^p \|(1 + |\cdot|^2)^{-d/2}\|_{L^1(\R^{d-1})}^p
= C_{d,p} r^{-p} \|g_1\|_{L^p(\R^{d-1})}^p,
\end{align*}
where we applied Young's inequality for convolutions.  Therefore,
\begin{align*}
\|I_k f(x_1, \cdot)\|_{L^p(\R^{d-1})} \leq C_{d,p} \int_0^\infty \frac{\|f(y_1, \cdot)\|_{L^p(\R^{d-1})}}{x_1+y_1} \ud y_1.
\end{align*}
Taking $L^p((0,\infty),w_{\gamma})$-norms in $x_1$ and applying Step I yields the required result.
\end{proof}

\begin{remark}
Actually, the kernel $k$ of Lemma \ref{lem:HardyHilbert} is a standard Calder\'on--Zygmund kernel, because $k$ is a.e. differentiable and
\[|\nabla_x k(x,y)|  + |\nabla_y k(x,y)| \leq |x-y|^{-d-1}, \ \ \ x\neq y. \]
Although we will not need it below let us note that
\cite[Corollary 2.10]{HH14} implies that $I_k$ is bounded on $L^p(\R^d,w)$ for any $w\in A_p$
\end{remark}

\section{Pointwise multiplication with $\one_{\R^d_+}$\label{sec:pointwise}}

In this section we prove the pointwise multiplier result, which is central in the characterization of the complex interpolation spaces of Sobolev spaces with boundary conditions in Section \ref{sec:intbdr}. Let $w_{\gamma}(x_1, \tilde{x}) = |x_{1}|^{\gamma}$, where $x_1\in \R$ and $\tilde{x}\in \R^{d-1}$.
 \begin{theorem}\label{thm:pointwise_multiplier}
Let $X$ be a UMD space, $p\in (1, \infty)$, $\gamma\in (-1,p-1)$, $\gamma' = -\gamma/(p-1)$, and assume $-\frac{\gamma'+1}{p'}<s <\frac{\gamma+1}{p}$. Then for all $f\in H^{s,p}(\R^{d},w_{\gamma};X) \cap L^{p}(\R^{d},w_{\gamma};X)$, we have $\one_{\R^{d}_{+}} f\in H^{s,p}(\R^{d},w_{\gamma};X)$  and
\[\|\one_{\R^{d}_{+}} f\|_{H^{s,p}(\R^{d},w_{\gamma};X)} \lesssim_{X,p,\gamma,s} \|f\|_{ H^{s,p}(\R^{d},w_{\gamma};X)},\]
and therefore, pointwise multiplication by $\one_{\R^{d}_{+}}$ extends to a bounded linear operator on $H^{s,p}(\R^{d},w_{\gamma};X)$.
\end{theorem}
To prove this the UMD property will only be used through the norm equivalence of Lemma \ref{lem:fractionalLaplace} below.

\begin{lemma}\label{lem:fractionalLaplace}
Let $X$ be a UMD space, $p\in (1, \infty)$, $s\in \R$, $\sigma\geq 0$, $w\in A_p$. Then
\[
(-\Delta)^{\sigma/2}: \mathcal{S}(\R^{d};X) \longrightarrow \mathcal{S}'(\R^{d};X),\,
f \mapsto \mathscr{F}^{-1}[(\xi \mapsto |\xi|^{\sigma})\wh{f}]
\]
defines for each $r \in \R$ (by extension by density) a bounded linear operator from $H^{r+\sigma,p}(\R^{d},w;X)$ to $H^{r,p}(\R^{d},w;X)$, independent of $r$ and $w$ (in the sense of compatibility), which we still denote by $(-\Delta)^{\sigma/2}$.
Moreover, $f \in H^{s+\sigma,p}(\R^{d},w;X)$ if and only if $f,(-\Delta)^{\sigma/2}f\in H^{s,p}(\R^{d},w;X)$, in which case
\[
\|f\|_{H^{s,p}(\R^{d},w;X)} \eqsim_{s,p,w,d,\sigma,X}
\|f\|_{H^{s-\sigma,p}(\R^{d},w;X)} + \|(-\Delta)^{\sigma/2}f\|_{H^{s-\sigma,p}(\R^{d},w;X)}.
\]
\end{lemma}
\begin{proof}
All assertions follow from the fact that the symbols
\[
\xi \mapsto \frac{|\xi|^{\sigma}}{(1+|\xi|^{2})^{2/\sigma}}, \quad
\xi \mapsto \frac{1}{(1+|\xi|^{2})^{2/\sigma}}, \quad
\xi \mapsto \frac{(1+|\xi|^{2})^{2/\sigma}}{1+|\xi|^{\sigma}}
\]
satisfy the conditions of Proposition \ref{prop:Mihlin}.
\end{proof}

In the proof of Theorem~\ref{thm:pointwise_multiplier} we will use the norm equivalence of the above lemma via (a variant of) a well known representation for $(-\Delta)^{\sigma/2}$ as a singular integral.
For $f \in H^{\sigma,p}(\R^{d})$ this representation reads as follows:
\[
(-\Delta)^{\sigma/2}f = \lim_{r \to 0^{+}}C_{d,\sigma} \int_{\R^{d} \setminus B(0,r)}\frac{T_{h}f-f}{h} \ud h,
\]
with limit in $L^{p}(\R^{d})$ (see \cite[Theorem 1.1(e)]{kwasnicki2015ten}); here $T_{h}$ denotes the left translation and $C_{d,\sigma}$ is a constant only depending on $d$ and $\sigma$.

In the proof we want to use a formula as above for $f$ replaced by $\one_{\R^{d}_{+}}f$,  which in general is an irregular function even if $f$ is smooth; in particular, a priori it is not clear that $\one_{\R^{d}_{+}}f \in H^{\sigma,p}(\R^{d})$.  We overcome this technical obstacle by Proposition~\ref{prop:eq_norm_differences} below, which provides a (non sharp) representation formula for $(-\Delta)^{\sigma/2}$ in spaces of distributions.

For the proof of Proposition~\ref{prop:eq_norm_differences} we need the following simple identity.
\begin{lemma}\label{lem:eq_norm_differences;simple_identity}
For each $\sigma \in (0,1)$ there exists a constant $c_{d,\sigma} \in (-\infty,0)$ such that
\[
|\xi|^{\sigma} = c_{d,\sigma} \int_{\R^{d}} \frac{e^{\imath h \cdot \xi}-1}{|h|^{d+\sigma}} \ud h,
\quad\quad \xi \in \R^{d}.
\]
Moreover, for all $\phi\in \mathcal{S}(\R^d)$
\begin{align}\label{eq:lem:eq_fractional_derivative&integral;claim_integral_formula_S'}
[\xi \mapsto |\xi|^{\sigma}](\phi) := \int_{\R^{d}}|\xi|^{\sigma}\phi(\xi) \ud\xi
& = c_{d,\sigma}\int_{\R^{d}}\int_{\R}\frac{e^{\imath h\xi}-1}{|h|^{d+\sigma}}\phi(\xi) \ud\xi \ud h
\\ & =:c_{d,\sigma}\int_{\R^{d}}\left[\xi \mapsto \frac{e^{\imath h\xi}-1}{|h|^{d+\sigma}}\right](\phi) \ud h.
\nonumber
\end{align}
\end{lemma}
\begin{proof}
Let $\xi \in \R^{d} \setminus \{0\}$ and choose $R \in \mathrm{O}(n)$ with $R\xi = |\xi|e_{1}$.
Then $h \cdot \xi = Rh \cdot R\xi = |\xi|Rh \cdot e_{1}$ and the substitution $y=|\xi|Rh$ yields
\[
\int_{\R^{d}} \frac{e^{\imath h \cdot \xi}-1}{|h|^{d+\sigma}} = |\xi|^{\sigma} \int_{\R^{d}}\frac{e^{\imath y_{1}}-1}{|y|^{d+\sigma}} \ud y.
\]
Observing that the integral on the right is a number in $(-\infty,0)$, the first identity follows.

Next we show \eqref{eq:lem:eq_fractional_derivative&integral;claim_integral_formula_S'}.
Given $\phi \in \mathcal{S}(\R^{d})$, the first identity gives
\[
[\xi \mapsto |\xi|^{\sigma}](\phi) = \int_{\R^{d}}|\xi|^{\sigma}\phi(\xi) \ud\xi
= c_{d,\sigma}\int_{\R^{d}}\int_{\R^{d}}\frac{e^{\imath h\xi}-1}{|h|^{d+\sigma}} \ud h\,\phi(\xi) \ud\xi.
\]
Since $\phi \in \mathcal{S}(\R^{d})$ and
\[
\frac{|e^{\imath h\xi}-1|}{|h|^{d+\sigma}} \leq 1_{|h| \leq 1}h^{-(d-1+\sigma)} |\xi| + 2\,\cdot\,1_{|h|>1}|h|^{-(d+\sigma)},
\]
we may invoke Fubini's theorem in order to get
\begin{align*}
[\xi \mapsto |\xi|^{\sigma}](\phi)
&= c_{d,\sigma}\int_{\R^{d}}\int_{\R}\frac{e^{\imath h\xi}-1}{|h|^{d+\sigma}}\phi(\xi) \ud\xi \ud h = c_{d,\sigma}\int_{\R^{d}}\left[\xi \mapsto \frac{e^{\imath h\xi}-1}{|h|^{d+\sigma}}\right](\phi) \ud h,
\end{align*}
as desired.
\end{proof}

For $f\in \Schw'(\R^d;X)$ let $\delta_{h} f= T_h f - f$, where $T_{h}$ denotes the left translation by $h$. For $0<r<R$ let $A(r,R) := \{ x \in \R^{d} : r < |x| < R \}$ be an annulus.

\begin{proposition}[Representation of $(-\Delta)^{\frac{\sigma}{2}}$]\label{prop:eq_norm_differences}
Let $p \in (1,\infty)$ and $\sigma \in (0,1)$.
For all $s \geq 0$ and $f \in H^{s,p}(\R^{d}) \otimes X \subset L^{p}(\R^{d};X)$ we have
\begin{equation*}
(-\Delta)^{\frac{\sigma}{2}} f = \frac{1}{c_{d,\sigma}}
\lim_{r \searrow 0, R \nearrow \infty} \left[ x \mapsto \int_{A(r,R)}\frac{\delta_{h}f(x)}{|h|^{d+\sigma}} \ud h \right]
\quad\quad \mbox{in} \quad H^{s-2,p}(\R^{d};X),
\end{equation*}
where $c_{d,\sigma}$ is the constant of Lemma \ref{lem:eq_norm_differences;simple_identity}.
\end{proposition}

The weights are left out on purpose, because translations are not well-behaved on weighted $L^p$-spaces. Moreover, no UMD is required  in the result above.

\begin{proof}
We prove this proposition by proving the following three statements:
\begin{enumerate}
\item The linear operator
\[
f \mapsto \left[ h \mapsto \frac{\delta_{h}f}{|h|^{d+\sigma}} \right]
\]
is bounded from $H^{s,p}(\R^{d};X)$ to $L^{1}(\R^{d};H^{s-2,p}(\R^{d};X))$ for all $s \in \R$ and thus gives rise to a bounded linear operator
\[
\mathcal{I}_{\sigma} : H^{s,p}(\R^{d};X) \longrightarrow H^{s-2,p}(\R^{d};X),\,f \mapsto \int_{\R^{d}}\frac{\delta_{h}f}{|h|^{d+\sigma}} \ud h,
\]

\item For all $s \geq 0$ we have
\[
\mathcal{I}_{\sigma} f =
\lim_{r \searrow 0, R \nearrow \infty} \left[ x \mapsto \int_{A(r,R)}\frac{\delta_{h}f(x)}{|h|^{d+\sigma}} \ud h \right]
\quad\quad \mbox{in} \quad H^{s-2,p}(\R^{d};X)
\]
for every $f \in H^{s,p}(\R^{d};X) \subset L^{p}(\R^{d};X)$.

\item For all $f\in H^{-\infty,p}(\R^{d})\otimes X$,
\begin{equation}\label{eq:identityfract}
\mathcal{I}_\sigma f = c_{d,\sigma}(-\Delta)^{\frac{\sigma}{2}} f \ \ \ \text{in} \ \ \ \Schw'(\R^d;X),
\end{equation}
where $c_{d,\sigma}$ is the constant of Lemma \ref{lem:eq_norm_differences;simple_identity}.
Here $H^{-\infty,p}(\R^{d}) = \bigcup_{s \in \R}H^{s,p}(\R^{d})$.
\end{enumerate}

(1): To prove this it is enough to establish the boundedness from $H^{s,p}(\R^{d};X)$ to $L^{1}(\R^{d};H^{s-2,p}(\R^{d};X))$.
As the Bessel potential operator $\mathcal{J}_{s}$ commutes with $\delta_{h}$, we may restrict ourselves to the case $s=2$.
Since by Lemma \ref{lem:compHW} $H^{2,p}(\R^{d};X)\hookrightarrow W^{1,p}(\R^d;X)$, we only need to estimate
\begin{equation}\label{eq:lem:eq_norm_differences;L1_estimate}
\int_{\R^{d}}\frac{\|\delta_{h}f\|_{L^p(\R^{d};X)}}{|h|^{d+\sigma}} \ud h \lesssim_{d,\sigma, p} \|f\|_{W^{1,p}(\R^{d};X)},
\quad\quad f \in W^{1,p}(\R^{d};X).
\end{equation}
To this end, let $f \in W^{1,p}(\R^{d};X)$.
Then
\[
\frac{\delta_{h}f}{h^{1+\sigma}}
= 1_{|h| \leq 1}|h|^{-(d-1+\sigma)}\int_{0}^{1} T_{th}\left[\nabla f \cdot \frac{h}{|h|}\right] \ud t + 1_{|h| > 1}|h|^{-(d+\sigma)}( T_{h}f-f ),
\]
where the integral is an $L^{p}(\R^{d};X)$-valued Bochner integral.
It follows that
\begin{align*}
\frac{\|\delta_{h}f\|_{L^p(\R^{d};X)}}{|h|^{d+\sigma}}
&\leq 1_{|h| \leq 1}|h|^{-(d-1+\sigma)} \int_{0}^{1}\| T_{th}\|\nabla f\|_{X^{d}}\|_{L^{p}(\R^{d})} \ud t \\
& \quad\quad +\: 1_{|h| > 1}|h|^{-(d+\sigma)}\left( \|T_{h}f\|_{L^{p}(\R;X)} + \|f\|_{L^{p}(\R;X)} \right) \\
&= 1_{|h| \leq 1}h^{-(d-1+\sigma)} \|\nabla f\|_{L^{p}(\R;X^{d})} +
2\,\cdot\,1_{|h| > 1}|h|^{-(d+\sigma)}\|f\|_{L^{p}(\R;X)}.
\end{align*}
Integrating over $h$ gives \eqref{eq:lem:eq_norm_differences;L1_estimate}.

(2): Let $s \geq 0$ and $f \in H^{s,p}(\R^{d};X) \subset L^{p}(\R^{d};X)$.
By the first assertion and the Lebesgue dominated convergence theorem,
\begin{equation}\label{eq:lem:eq_norm_differences;convergence_integrals}
\mathcal{I}_{\sigma} f =
\lim_{r \searrow 0, R \nearrow \infty} \int_{A(r,R)}\frac{\delta_{h}f}{|h|^{d+\sigma}} \ud h
\quad\quad \mbox{in} \quad H^{s-2,p}(\R^{d};X),
\end{equation}
where the integrals $\int_{A(r,R)}\frac{\delta_{h}f}{|h|^{d+\sigma}} \ud h $ are Bochner integrals in $H^{s-2,p}(\R^{d};X)$.
As $f \in L^{p}(\R^{d};X)$, $h \mapsto \frac{\delta_{h}f}{|h|^{d+\sigma}}$ is in $L^{1}(A(r,R);L^{p}(\R^{d};X))$ for every $0<r<R<\infty$. Since $L^{p}(\R^{d};X),H^{s-2,p}(\R^{d};X) \hookrightarrow \mathcal{S}'(\R^{d};X)$, it follows that the integrals $\int_{A(r,R)}\frac{\delta_{h}f}{|h|^{d+\sigma}} \ud h $ in \eqref{eq:lem:eq_norm_differences;convergence_integrals} can also be considered as Bochner integrals in $L^{p}(\R^{d};X)$, implying that $\int_{A(r,R)}\frac{\delta_{h}f}{|h|^{d+\sigma}} \ud h = \left[ x \mapsto \int_{A(r,R)}\frac{\delta_{h}f(x)}{|h|^{d+\sigma}} \ud h \right]$ (see \cite[Proposition~1.2.25]{HNVW1}).

(3) By linearity it suffices to consider the scalar case $f\in H^{s,p}(\R^d)$ for some $s\in \R$. By the density of $\mathcal{S}(\R^{d})\subseteq H^{s,p}(\R^d)$ (see Lemma \ref{lem:densityCcH}) it suffices to consider $f\in \mathcal{S}(\R^{d})$. Indeed, this follows from the boundedness of $\mathcal{I}_{\sigma}$ and $(-\Delta)^{\sigma/2}$ (see (1). Now \eqref{eq:identityfract} follows from well-known results (see \cite[Theorem 1.1(e)]{kwasnicki2015ten}). For convenience we include a direct proof.
Using Lemma \ref{lem:eq_norm_differences;simple_identity}, for each $f \in \mathcal{S}(\R^{d};X)$ we find
\begin{align*}
(-\Delta)^{\sigma/2}f
&= \mathscr{F}^{-1}[(\xi \mapsto |\xi|^{\sigma})\wh{f}]
 = \mathscr{F}^{-1}\left[ c_{d,\sigma}\int_{\R^{d}}\left[ \xi \mapsto \frac{e^{\imath h\xi}-1}{|h|^{d+\sigma}}\wh{f}(\xi) \right] \ud h  \right] \\
&= c_{d,\sigma}\int_{\R^{d}} \mathscr{F}^{-1}\left[ \xi \mapsto \frac{e^{\imath h\xi}-1}{|h|^{d+\sigma}}\wh{f}(\xi) \right] \ud h
 = c_{d,\sigma}\int_{\R^{d}} \frac{\delta_{h}f}{|h|^{d+\sigma}} \ud h,
\end{align*}
where all integrals are in $\mathcal{S}'(\R^{d};X)$.
By (1), for every $f \in \mathcal{S}(\R^{d};X) \subset H^{0,p}(\R^{d};X)$ we have $\mathcal{I}_{\sigma}f = \int_{\R^{d}} \frac{\delta_{h}f}{|h|^{d+\sigma}} \ud h$, where the integral is taken in $H^{-1,p}(\R^{d};X) \hookrightarrow \mathcal{S}'(\R^{d};X)$.
This proves \eqref{eq:identityfract}, as desired.
\end{proof}

Finally we are in position to prove the pointwise multiplier result.

\begin{proof}[Proof of Theorem \ref{thm:pointwise_multiplier}]
We only consider $s\geq 0$. The case $s<0$ follows from a duality argument using \cite[Proposition 3.5]{MeyVerpoint}.

By Lemma \ref{lem:densityCcH} it is enough to prove $\|\one_{\R^{d}_{+}}f\|_{H^{s,p}(\R^{d},w_{\gamma};X)} \lesssim_{s,p,d,\gamma,X} \|f\|_{H^{s,p}(\R^{d},w_{\gamma};X)}$ for an arbitrary $f \in \mathcal{S}(\R^{d})\otimes X$. Let $g:=\one_{\R^{d}_{+}}f \in L^{p}(\R^{d})\otimes X$. By Lemma \ref{lem:fractionalLaplace}, we have
\[\|g\|_{H^{s,p}(\R^{d},w_{\gamma};X)}\lesssim_{s,p,d,\gamma,X} \|g\|_{L^p(\R^{d},w_{\gamma};X)} + \|(-\Delta)^{s/2}g\|_{L^{p}(\R^{d},w_{\gamma};X)}.\]
Clearly, $\|g\|_{L^p(\R^{d},w_{\gamma};X)}\leq \|f\|_{L^p(\R^{d},w_{\gamma};X)}$
from which we see that it suffices to show
\begin{equation}\label{eq:suffestg}
\|(-\Delta)^{s/2}g\|_{L^{p}(\R^{d},w_{\gamma};X)} \lesssim_{s,p,d,\gamma} \|f\|_{H^{s,p}(\R^{d},w_{\gamma};X)}.
\end{equation}
By Proposition~\ref{prop:eq_norm_differences},
\[
\mathcal{I}_{s,j}g  :=
\left[ x \mapsto \int_{A(\frac{1}{j},j)}\frac{\delta_{h}g(x)}{|h|^{d+s}} \ud h \right] \stackrel{j \to \infty}{\longrightarrow} (-\Delta)^{s/2}g
\quad\quad \mbox{in} \quad H^{s-2,p}(\R^{d};X) \hookrightarrow \mathcal{S}'(\R^{d};X).
\]
In order to finish the proof, it is thus enough to show that $\mathcal{I}_{s,j}g$ converges in $L^{p}(\R^{d},w_{\gamma};X) + L^{p}(\R^{d};X) \hookrightarrow \mathcal{S}'(\R^{d};X)$ to some $G$ satisfying
\begin{equation}\label{eq:toproveGsum}
\|G\|_{L^{p}(\R^{d},w_{\gamma};X)} \lesssim_{s,p,d,\gamma,X} \|f\|_{H^{s,p}(\R^{d},w_{\gamma};X)}.
\end{equation}
Indeed, then $(-\Delta)^{s/2}g = G$ and \eqref{eq:suffestg} holds.

Defining
\[
S := \left\{ (y,z) \in \R^{2} : [ z < -y \:\mbox{and}\: y>0] \:\:\mbox{or}\:\: [ z > -y \:\mbox{and}\: y<0] \right\}
\]
we have
\begin{align}
\mathcal{I}_{s,j}g
&= G_{1,j} + G_{2,j} \nonumber \\
&:= \one_{\R^{d}_{+}}\mathcal{I}_{s,j}f + \left[ x \mapsto -\mathrm{sgn}(x_{1})\int_{A(\frac{1}{j},j)}\one_{S}(x_{1},h_{1})\frac{f(x+h)}{|h|^{d+s}} \ud h \right],
\label{eq:thm:pointwise_multiplier;formula_case_split}
\end{align}
where $\mathcal{I}_{s,j}f$ is defined analogously to $\mathcal{I}_{s,j}g$:
\[
\mathcal{I}_{s,j}f  :=
\left[ x \mapsto \int_{A(\frac{1}{j},j)}\frac{\delta_{h}f(x)}{|h|^{d+s}} \ud h \right].
\]

We first consider $\{G_{1,j}\}_{j \in \N}$.
Since $\mathcal{I}_{s,j}f  \stackrel{j \to \infty}{\longrightarrow} (-\Delta)^{s/2}f$ in $L^{p}(\R^{d};X)$ by Proposition~\ref{prop:eq_norm_differences},
it follows that $G_{1} := \one_{\R^{d}_{+}}(-\Delta)^{s/2}f = \lim_{j \to \infty}G_{1,j}$ in $L^{p}(\R^{d};X)$.
By Proposition Lemma~\ref{lem:fractionalLaplace},
\begin{align*}
\|G_{1}\|_{L^{p}(\R^{d},w_{\gamma};X)} & \leq \|(-\Delta)^{s/2}f\|_{L^{p}(\R^{d},w_{\gamma};X)} \lesssim_{s,p,d,\gamma,X} \|f\|_{H^{s,p}(\R^{d},w_{\gamma};X)}.
\end{align*}

We next consider $\{G_{2,j}\}_{j \in \N}$.
Observing that
\[
|h| = (|h_{1}|^{2}+|\tilde{h}|^{2})^{1/2} = ((|t|+|h_{1}+t|)^{2}+|\tilde{h}|^{2})^{1/2}
\]
for all $h=(h_{1},\tilde{h}) \in \R^{d}$ and $t \in \R$ with $(t,h_{1}) \in S$, we find
\begin{align*}
\int_{A(\frac{1}{j},j)}\one_{S}(x_{1},h_{1})\frac{\|f(x+h)\|_{X}}{|h|^{d+s}} \ud h
&\leq \int_{\R^{d}}\frac{\|f(x+h)\|_{X}}{((|x_{1}|+|h_{1}+x_{1}|)^{2}+|\tilde{h}|^{2})^{\tfrac{d+s}{2}}} \ud h \\
&= \int_{\R^{d}}\frac{\|f(y)\|_{X}}{((|x_{1}|+|y_{1}|)^{2}+|\tilde{y}-\tilde{x}|^{2})^{\tfrac{d+s}{2}}} \ud y \\
&\leq \int_{\R^{d}}k(x,y)|y_{1}|^{-s}\|f(y)\|_{X} \ud y,
\end{align*}
where $k(x,y) = ((|x_{1}|+|y_{1}|)^{2}+|\tilde{y}-\tilde{x}|^{2})^{\tfrac{d}{2}}$.
Applying Lemma~\ref{lem:HardyHilbert} to the function $\phi(y) = |y_{1}|^{-s} \|f(y)\|_{X}$
we thus obtain
\begin{align*}
\Big\| x \mapsto \int_{A(\frac{1}{j},j)}\one_{S}(x_{1},h_{1})\frac{\|f(x+h)\|_{X}}{|h|^{d+s}} \ud h \Big\|_{L^{p}(\R^{d},w_{\gamma})}
&\leq \|I_k \phi\|_{L^p(\R^{d},w_{\gamma})} \\
&\lesssim_{p,d,\gamma} \|\phi\|_{L^p(\R^{d},w_{\gamma})} \\
&= \|f\|_{L^p(\R^{d},w_{\gamma-s p};X)}.
\\ & \lesssim_{p,d,\gamma} \|f\|_{H^{s,p}(\R^{d},w_{\gamma};X)},
\end{align*}
where in the last step we applied Lemma \ref{lem:Hardypp}. It follows that the limit $G_{2} := \lim_{j \to \infty}G_{2,j}$ exists in $L^{p}(\R^{d},w_{\gamma};X)$ and, moreover,
\[\|G_{2}\|_{L^{p}(\R^{d},w_{\gamma};X)} \lesssim_{p,d,\gamma} \|f\|_{H^{s,p}(\R^{d},w_{\gamma};X)}.\]

Finally, combining the just obtained results for $\{G_{1,j}\}_{j \in \N}$ and $\{G_{2,j}\}_{j \in \N}$, we see that
$G:= G_{1} + G_{2} = \lim_{j \to \infty }\mathcal{I}_{s,j}g$ in $L^{p}(\R^{d},w_{\gamma};X) + L^{p}(\R^{d};X) \hookrightarrow \mathcal{S}'(\R;X)$
and \eqref{eq:toproveGsum} holds as desired.
\end{proof}

\section{Interpolation theory without boundary conditions}\label{sec:interpo}

For details on interpolation theory we refer the reader to \cite{BeLo,Tr1}. In this section we present some weighted and vector-valued versions of known results.

The following extension operator will allow us to reduce the half space case $\R^{d}_+$ to the full space $\R^d$.
\begin{lemma}[Extension operator]\label{lem:extension}
Let $X$ be a Banach space. Let $p\in (1, \infty)$, and $m\in \N_0$.
Let $w\in A_p$ be such that $w(-x_1,\tilde{x}) = w(x_1, \tilde{x})$ for $x_1\in \R$ and $\tilde{x}\in \R^{d-1}$. Then there exists an operator $\mathcal{E}_+^m:L^{p}(\R^d_+,w;X)\to L^{p}(\R^d,w;X)$ such that
\begin{enumerate}
\item For all $f\in L^p(\R^d_+,w;X)$, $(\mathcal{E}_+^m f)|_{\R^d_+} = f$;
\item for all $k\in \{0,\ldots, ,m\}$, $\mathcal{E}_+^m:W^{k,p}(\R_+^d,w;X)\to W^{k,p}(\R^d,w;X)$ is bounded,
\end{enumerate}
Moreover, if $f\in L^p(\R^d_+,w;X)\cap C^{m}(\R^d_+;X)$, then $\mathcal{E}_+^m f$ is $m$-times continuous differentiable on $\R^d$.
\end{lemma}
By a reflection argument the same holds for $\R^d_-$. The corresponding operator will be denoted by $\mathcal{E}_-^m$.
\begin{proof}
The result is a simple extension of the classical construction given in \cite[Theorem 5.19]{AF03} to the weighted setting.
The final assertion is clear from the construction of $\mathcal{E}_+^m$.
\end{proof}

To define Bessel potential spaces on domains, we proceed in an abstract way using factor spaces.
\begin{definition}\label{def:factor}
Let $\mathbb{F} \hookrightarrow \Distr'(\R^d;X)$ be a Banach space.
Define the {\em restricted space/factor space} to an open set $\Omega\subseteq \R^d$ as
\[\mathbb{F}(\Omega) := \{ f \in \Distr'(\R^{d};X) : \exists g\in \mathbb{F}, f = g|_{\Omega} \} \] and let
\[\|f\|_{\mathbb{F}(\Omega)} = \inf\{\|g\|_{\mathbb{F}}: g|_{\Omega} = f\}.\]
We say that $\mathcal{E}$ is an {\em extension operator for $\mathbb{F}(\Omega)$} if
\begin{enumerate}
\item for all $f\in \mathbb{F}(\Omega)$, $(\mathcal{E} f)|_{\Omega} = f$;
\item $\mathcal{E}:\mathbb{F}(\Omega)\to \mathbb{F}$ is bounded.
\end{enumerate}
\end{definition}

For $p \in (1,\infty)$, $w \in A_{p}$ and an open set $\Omega \subset \R^{d}$, we define the Bessel potential space $H^{s,p}(\Omega,w;X)$ as the factor space
\[
H^{s,p}(\Omega,w;X) := [H^{s,p}(\R^{d},w;X)](\Omega).
\]

By Lemma \ref{lem:extension} and for $w$ as stated there, we find that $W^{k,p}(\R_+^d,w;X)$ can be identified (up to an equivalent norm) with the factor space $[W^{k,p}(\R^{d},w;X)](\R^{d}_{+})$, where an extension operator can also be found.
Indeed, let $W^{k,p}_{\rm factor}(\R_+^d,w;X) = [W^{k,p}(\R^{d},w;X)](\R^{d}_{+})$ denote the factor space. For $f\in W^{k,p}_{\rm factor}(\R^d,w;X)$ let $g\in W^{k,p}(\R^d,w;X)$ be such that $g|_{\R^d_+} = f$. Then
\[\|f\|_{W^{k,p}(\R_+^d,w;X)} \leq \|g\|_{W^{k,p}(\R^d,w;X)}.\]
Taking the infimum over all of the above $g$, we find
\[\|f\|_{W^{k,p}(\R_+^d,w;X)} \leq \|f\|_{W^{k,p}_{\rm factor}(\R^d_+,w;X)}.\]
Next let $f\in W^{k,p}(\R_+^d,w;X)$. Then $\mathcal{E}_+ f\in W^{k,p}(\R^d,w;X)$ and
\[\|f\|_{W^{k,p}_{\rm factor}(\R^d_+,w;X)} \leq \|\mathcal{E}_+ f\|_{W^{k,p}(\R^d,w;X)} \leq C\|f\|_{W^{k,p}(\R^d_+,w;X)}.\]

Next we present two abstract lemmas to identify factor spaces in the complex interpolation scale. The result is a straightforward consequence of \cite[Theorem~1.2.4]{Tr1}. We include the short in order to be able to track the constants. For details on complex interpolation theory we refer to \cite[Section~1.9.3]{Tr1}.

\begin{lemma}\label{lem:coretract}
Let $(X_0, X_1)$ and $(Y_0, Y_1)$ be interpolation couples and let $X_{\theta} = [X_0, X_1]_{\theta}$ and $Y_{\theta} = [Y_0, Y_1]_{\theta}$ for a given $\theta\in (0,1)$. Assume $R:X_0+X_1\to Y_0+Y_1$ and $S:Y_0+Y_1\to X_0+X_1$ are linear operators such that $S\in \calL(Y_j, X_j)$, $R\in \calL(X_j, Y_j)$ and $RS$ is the identity operator on $Y_j$ for $j\in \{0,1\}$. Then
$SR$ defines a projection on $X_{\theta}$ and
$R$ is an isomorphism from $SR(X_\theta)$ onto $Y_{\theta}$ with inverse $S$. Moreover, the following estimates hold:
\begin{align*}
C_S^{-1} \|Sy\|_{X_{\theta}} \leq \|y\|_{Y_{\theta}}  & \leq C_R \|Sy\|_{X_{\theta}}, \ \ \ y\in Y_{\theta},
\\ \|Rx\|_{Y_{\theta}} &\leq C_R \|x\|_{X_{\theta}}, \ \ \ x\in X_{\theta},
\\ \|x\|_{X_{\theta}} &\leq C_S \|Rx\|_{Y_{\theta}}, \ \ \ x\in SR(X_{\theta}),
\end{align*}
where $C_R = \max_{j\in \{0,1\}} \|R\|_{\calL(X_j,Y_j)}$ and $C_S = \max_{j\in \{0,1\}} \|S\|_{\calL(X_j,Y_j)}$.
\end{lemma}
\begin{proof}
By complex interpolation we know
\[\|S\|_{\calL(Y_\theta,X_\theta)}\leq C_S, \ \ \text{and}  \ \   \|R\|_{\calL(X_\theta,Y_\theta)}\leq C_R\]
and $RS$ is the identity operator on $Y_{\theta}$. This proves the upper estimates for $S$ and $R$. To see that $SR$ is a projection note that $(SR)(SR) = SR$. The lower estimate for $S$ follows from
\begin{align*}
\|y\|_{Y_{\theta}} & = \|R Sy\|_{Y_{\theta}} \leq C_{R} \|Sy\|_{X_{\theta}}, \ \ \ y\in Y_{\theta}.
\end{align*}
To prove the lower estimate for $R$ note that for $x := SR u\in SR(X_{\theta})$
\[\|x\|_{X_{\theta}} =  \|S R SR u\|_{X_{\theta}}  \leq C_S \|R SR u\|_{Y_{\theta}} = C_S \|R x\|_{X_{\theta}}.\]
\end{proof}

\begin{lemma}\label{lem:abstractinterpdomain}
Let $\mathbb{F}^0,\mathbb{F}^1 \hookrightarrow \Distr'(\R^d;X)$ be two Banach spaces. For $\theta\in (0,1)$, let
\[\mathbb{F}^\theta = [\mathbb{F}^0, \mathbb{F}^1]_{\theta}.\]
Let $\Omega\subseteq \R^d$ be an open set, and define $\mathbb{F}^\theta(\Omega)$ as in Definition \ref{def:factor}, and assume there is an extension operator $\mathcal{E}$ for $\mathbb{F}^s(\Omega)$ for $s\in \{0,1\}$. Then
$[\mathbb{F}^0(\Omega), \mathbb{F}^1(\Omega)]_{\theta} = \mathbb{F}^\theta(\Omega)$ and
\[ C^{-1} \|f\|_{\mathbb{F}^\theta(\Omega)} \leq\|f\|_{[\mathbb{F}^0(\Omega), \mathbb{F}^1(\Omega)]_{\theta}} \leq \|f\|_{\mathbb{F}^\theta(\Omega)}\]
where $C$ only depends on the norms of the extension operator. Moreover, $\mathcal{E}$ is an extension operator for $\mathbb{F}^\theta(\Omega)$.
\end{lemma}
\begin{proof}
Define $R: \mathbb{F}^j\to \mathbb{F}^j(\Omega)$ by $R f = f|_{\Omega}$ and $S:\mathbb{F}^j(\Omega)\to \mathbb{F}^j$ as $S = \mathcal{E}$. Then $\|R\|\leq 1$,$\|S\|\leq C$ and $RS = I$. From Lemma \ref{lem:coretract} we conclude that for all $f\in [\mathbb{F}^0(\Omega), \mathbb{F}^1(\Omega)]_{\theta}$
\[C^{-1}\|f\|_{\mathbb{F}^\theta(\Omega)} \leq C^{-1}\|\mathcal{E} f\|_{\mathbb{F}^\theta}\leq \|f\|_{[\mathbb{F}^0(\Omega), \mathbb{F}^1(\Omega)]_{\theta}}.\]

Conversely, let $f\in \mathbb{F}^{\theta}(\Omega)$. Choose, $g\in \mathbb{F}^{\theta}$ such that $R g = g|_{\Omega} = f$. Since $\|R\|\leq 1$, by complex interpolation we find
\[\|f\|_{[\mathbb{F}^0(\Omega), \mathbb{F}^1(\Omega)]_{\theta}}\leq \|g\|_{[\mathbb{F}^0, \mathbb{F}^1]_{\theta}} = \|g\|_{\mathbb{F}^{\theta}}\]
Taking the infimum over all $g$ as above, the result follows.

To show the final assertion, note that $\mathcal{E}\in \calL(\mathbb{F}^{\theta}(\Omega), \mathbb{F}^{\theta})$ by the above. Moreover, for $f\in \mathbb{F}^0(\Omega)\cap \mathbb{F}^1(\Omega)$,   $(\mathcal{E}f)|_{\Omega} = f$. By density (see \cite[Theorem 1.9.3]{Tr1}) this extends to all $f\in \mathbb{F}^{\theta}(\Omega)$.
\end{proof}

\begin{proposition}\label{prop:W=H2}
Let $X$ be a UMD space, $p\in (1, \infty)$, $k\in \N_0$ and assume $w\in A_p$ is such that $w(x_1, \tilde{x}) = w(-x_1, \tilde{x})$ for $x_1\in \R$ and $\tilde{x}\in \R^{d-1}$. Then $H^{k,p}(\R^d_+,w;X) = W^{k,p}(\R^d_+,w;X)$
\end{proposition}
\begin{proof}
This is immediate from Proposition \ref{prop:W=H} and the fact that $W^{k,p}(\R^d_+,w;X)$ coincides with the factor space $[W^{k,p}(\R^d,w;X)](\R^{d}_{+})$.
\end{proof}

Next we identify the complex interpolation spaces of $H^{s,p}(\Omega,w;X)$. Here the UMD property is needed to obtain bounded imaginary powers of $-\Delta$.
\begin{proposition}\label{prop:extensionH}
Let $X$ be a UMD space and $p\in (1, \infty)$.
Let $w\in A_p$ be such that $w(-x_1, \tilde{x}) = w(x_1, \tilde{x})$ for all $x_1\in \R$ and $\tilde{x}\in \R^{d-1}$.
\begin{enumerate}[$(1)$]
\item\label{it:extensionH1} Let $\theta\in [0,1]$ and $s_0, s_1,s\in \R$ be such that $s = s_0 (1-\theta) + s_1 \theta$. Then for $\Omega = \R^d$ or $\Omega = \R^d_+$ one has
\[[H^{s_0,p}(\Omega,w;X), H^{s_1,p}(\Omega,w;X)]_{\theta} = H^{s,p}(\Omega,w;X)\]
\item\label{it:extensionH2} For each $m\in \N_0$ there exists an $\mathcal{E}_+^m\in \calL(H^{-m,p}(\R^d_+,w;X), H^{-m,p}(\R^d,w;X))$ such that
\begin{itemize}
\item for all $|s|\leq m$, $\mathcal{E}_+\in \calL(H^{s,p}(\R_+^d,w;X), H^{s,p}(\R^d,w;X))$,
\item for all $|s|\leq m$, $f\mapsto (\mathcal{E}_+ f)|_{\R^d_+}$ equals the identity operator on $H^{s,p}(\R_+,w;X)$.
    \end{itemize}
Moreover, if $f\in L^p(\R^d_+,w;X)\cap C^{m}(\overline{\R^d_+};X)$, then $\mathcal{E}_+^m f\in C^{m}(\R^d;X)$.
\end{enumerate}
\end{proposition}
By a reflection argument the same holds for $\R^d_-$. The corresponding operator will be denoted by $\mathcal{E}_-^m$.
\begin{proof}
(1): For $\Omega= \R^d$, the result follows from \cite[Proposition 3.2 and 3.7]{MeyVerpoint} (see \cite[Theorem 5.6.9]{HNVW1} for the unweighed case).

(2): Fix $m\in \N$. We first construct $\wtE\in \calL(H^{-m,p}(\R^d,w;X))$ such that
\begin{enumerate}[(i)]
\item $\wtE\in \calL(H^{s,p}(\R^d,w;X))$ for all $|s|\leq m$;
\item $\wtE f|_{\R_+^d} = f|_{\R^d_+}$;
\item $\wtE f = 0$ if $f|_{\R_+^d} = 0$;
\end{enumerate}
Given $\wtE$ we can define $\mathcal{E}^m_+:H^{s,p}(\R_+^d,w;X)\to H^{s,p}(\R^d,w;X)$
by $\mathcal{E}^m_+ f = \wtE \wt{f}$ where $\wt{f}\in H^{s,p}(\R^d,w;X)$ satisfies $\wt{f}|_{\R^d_+} = f$. This is well-defined by (iii).

In order to construct $\wtE$ let $0<\lambda_1<\ldots < \lambda_{2m+2}<\infty$ and $b_1, \ldots, b_{2m+2}\in \R$ be as in \cite[2.9.3]{Tr1}. For $\lambda\in \R\setminus\{0\}$ we write $T_{\lambda} f(x) = f(-\lambda x_1, \tilde{x})$. Let $\wtE\in \calL(L^p(\R^d,w;X))$ and $\wtEr\in \calL(L^{p'}(\R^d,w';X^*))$ be defined by \begin{align*}
\wtE f &= \one_{\R^d_+} f + \one_{\R^d_-} \sum_{j=1}^{2m+2} b_j T_{\lambda_j}f, \qquad  \wtEr g
= \one_{\R^d_+} \Big(g + \sum_{j=1}^{2m+2} b_j \lambda_j^{-1} T_{\lambda_j^{-1}}g\Big).
\end{align*}
Then one can check that
\begin{equation}\label{eq:wtEduality}
\lb \wtE f,g\rb = \lb  f,\wtEr g\rb, \ \ \ f\in L^p(\R^d,w;X),  \ \ g\in L^{p'}(\R^d,w';X^*).
\end{equation}
Moreover, by the special choice of $b_1, \ldots, b_{2m+2}$ it is standard to check that $\wtE\in \calL(W^{m,p}(\R^d,w;X))$ and $\wtEr\in \calL(W^{m,p'}(\R^d,w';X^*))$. In view of \eqref{it:extensionH1} for $\Omega = \R^d$ and Proposition \ref{prop:W=H}, complex interpolation gives $\wtE\in \calL(H^{s,p}(\R^d,w;X))$ and $\wtEr\in \calL(H^{s,p'}(\R^d,w';X^*))$ for all $0\leq s\leq m$.

Recall that $H^{s,p}(\R^d,w;X)= (H^{-s,p'}(\R^d,w';X^*))^*$ (see \cite[Proposition 3.5]{MeyVerpoint}), $X$ being reflexive as a UMD space (see \cite[Theorem 4.3.3]{HNVW1}). By the duality relation \eqref{eq:wtEduality} we find that $\wtE$ extends to a bounded linear operator on $H^{s,p}(\R^d,w;X)$ for each $s\in [-m,0]$. Therefore, (i) follows and moreover (ii) follows by a density argument.
To check (iii) let $f\in H^{-m,p}(\R^d,w;X)$ with $f|_{\R_+^d} = 0$ be given.  Let $\phi\in C^\infty_c(\R^d_-)$ be such that $\int \phi \ud x  =1$ and set $\phi_n:= n^{-d} \phi(n\cdot)$ for $n\in \N$. Then, by Lemma \ref{lem:convolution;Bessel_pot_space}, $\phi_n*f\to f$ in $H^{-m,p}(\R^d,w;X)$ and $\phi_n*f\in L^p(\R^d,w;X)$. Now since $\phi_n*f|_{\R^d_+} = 0$ it follows that $\wtE f|_{\R^d_+} = \limn \wtE \phi_n*f|_{\R^d_+} = 0$.

Finally, note that for $f\in L^p(\R^d_+,w;X)\cap C^{m}(\overline{\R^d_+};X)$, $\mathcal{E}^+_m f \in C^{m}(\overline{R}^{d}_{-};X) \oplus C^{m}(\overline{R}^{d}_{+};X)$ with
\[\mathcal{E}^+_m f|_{\R^d_+} =  f \ \ \ \text{and} \ \ \  \mathcal{E}^+_m f|_{\R^d_-} = \sum_{j=1}^{2m+2} b_j T_{\lambda_j}f\]
and by the special choice of $b_1, \ldots, b_{2m+2}$, one can check that $f\in C^{m}(\R^d;X)$.

Now (1) for $\Omega = \R^d_+$ follows from Lemma \ref{lem:abstractinterpdomain} and (2).
\end{proof}

For an open set $\Omega\subseteq \R^d$, and $s\in \R$ let $H^{s,p}_{\Omega}(\R^d,w_{\gamma};X)$ be the closed subspace of $H^{s,p}(\R^d,w_{\gamma};X)$ of functions with support in $\overline{\Omega}$.

\begin{proposition}\label{prop:interpH}
Let $X$ be a UMD space, $p\in (1, \infty)$, $k\in \N$, $w(-x_1, \tilde{x}) = w(x_1, \tilde{x})$ for all $x_1\in \R$ and $\tilde{x}\in \R^{d-1}$. Let $\theta\in [0,1]$ and $s_0, s_1,s\in \R$ be such that $s = s_0 (1-\theta) + s_1 \theta$.
Then the following identity holds with equivalence of norms
\[[H^{s_0,p}_{\R_{\pm}^d}(\R^d,w;X), H^{s_1,p}_{\R_{\pm}^d}(\R^d,w;X)]_{\theta} = H^{s,p}_{\R_{\pm}^d}(\R^d,w;X).\]
\end{proposition}
\begin{proof}
To show this we consider the case of $\R_+^d$. The other case can be proved in the same way. Let $\mathcal{E}_{-}^m$ be the (reflected) extension operator of Proposition~\ref{prop:extensionH} with $m$ the least integer above $\max\{|s_0|, |s_1|\}$.
Define $R:H^{s_{0} \wedge s_{1},p}(\R^{d},w;X) \to H^{s_{0} \wedge s_{1},p}_{\R^{d}_{+}}(\R^{d},w;X)$ by
\[R f := f- \mathcal{E}^{m}_{-} (f|_{\R_-^d})\]
and let $S:H^{s_{0} \wedge s_{1},p}_{\R^d_+}(\R^d,w;X)\to H^{s_{0} \wedge s_{1},p}(\R^d,w;X)$ be the inclusion operator.
For each $t \in [s_{0} \wedge s_{1},m]$, $R$ and $S$ restrict to bounded linear operators $R:H^{t,p}(\R^{d},w;X) \to H^{t,p}_{\R^{d}_{+}}(\R^{d},w;X)$ and $S:H^{t,p}_{\R^d_+}(\R^d,w;X)\to H^{t,p}(\R^d,w;X)$ with the property that $SR(H^{t,p}(\R^d,w;X))=H^{t,p}_{\R^d_+}(\R^d,w;X)$.
Using Lemma~\ref{lem:coretract} in combination with Proposition~\ref{prop:extensionH} we find that $R$ restricts to an isomorphism from $H^{s,p}_{\R^d_+}(\R^d,w;X) = SR(H^{s,p}(\R^d,w;X))$ to $[H^{s_0,p}_{\R_{+}^d}(\R^d,w;X), H^{s_1,p}_{\R_{+}^d}(\R^d,w;X)]_{\theta}$.
Since $R f = f$ for all $f\in H^{s,p}_{\R^d_+}(\R^d,w;X)$, this proves the required identity for the interpolation space. The norm equivalence follows from the estimates in Lemma~\ref{lem:coretract} as well.
\end{proof}

To end this section we present a variation of a classical interpolation inequality.
The result can be deduced from the weighted Gagliardo-Nirenberg type inequality \cite[Proposition 5.1]{MeyVersharp}. We provide a more direct proof which also yields additional information. The unweighted and scalar-valued case can be found in \cite[Theorem 1.5.1]{Kry08}. However, the proof given there does not extend to the weighted setting. The lemma can also be deduced from Proposition \ref{prop:Mihlin}, but this would require $X$ to be a UMD space (cf.\ the proof of \cite[Corollary 5.3]{GV17}).

\begin{lemma}[Gagliardo-Nirenberg inequality]\label{lem:normequivalenceSobolevhighestorder}
Let $X$ be a Banach space and $k\in \N$. Let $\Omega = \R^d$ or $\Omega = \R_+^d$. Let $w\in A_p$ be such that $w(-x_1, \tilde{x}) = w(x_1, \tilde{x})$ if $\Omega = \R^d_+$. Then for all $u\in W^{k,p}(\Omega,w;X)$ and $j\in \{1, \ldots, k-1\}$,
\[[u]_{W^{j,p}(\Omega,w;X)} \lesssim_{p,k,[w]_{A_p}} \|u\|_{L^p(\Omega,w;X)}^{1-\frac{j}{k}}[u]_{W^{k,p}(\Omega,w;X)}^{\frac{j}{k}}.\]
\end{lemma}
\begin{proof}
By an iteration argument one sees that it suffices to consider $j=1$ and $k=2$ (see \cite[Exercise 1.5.6]{Kry08}).

First consider the case $\Omega = \R^{d}$.  For $u\in W^{2,p}(\R^d,w;X)$, it follows from Lemma~\ref{lem:compHW} that
\begin{align*}
[u]_{W^{1,p}(\R^d,w;X)} & \leq \|u\|_{W^{1,p}(\R^d,w;X)}
\\ & \lesssim_{p,[w]_{A_p}} \|u\|_{H^{2,p}(\R^d,w;X)}\leq  \|u\|_{L^p(\R^d,w;X)} + [u]_{W^{2,p}(\R^d,w;X)}.
\end{align*}
For $\lambda > 0$ let $u_{\lambda}(x) = u(\lambda x)$ and $w_{\lambda} = w(\lambda x)$ and note that $[w]_{A_p} = [w_{\lambda}]_{A_p}$. Then applying the estimate to $u_{\lambda}$ and the weight $w_{\lambda}$, a substitution yields
\[[u]_{W^{1,p}(\R^d,w;X)} \lesssim_{p,n,[w]_{A_p}} \lambda^{-1}\|u\|_{L^p(\R^d,w;X)} + \lambda [u]_{W^{2,p}(\R^d,w;X)}.\]
Minimizing over $\lambda>0$ the result follows.

In the case $\Omega = \R_+^d$ we use a standard extension argument. Let $\mathcal{E}^2_+$ be the extension operator from Lemma \ref{lem:extension}. Then by \cite[Theorem 5.19]{AF03}, $\mathcal{E}_+^2$ has the following additional property: for all $|\alpha|\leq 2$, $\partial^{\alpha} \mathcal{E}^2_+ = E_{\alpha} \partial^{\alpha}$, where $E_\alpha$ is an extension operator for $W^{2-|\alpha|}(\R^d_+,w;X)$. Therefore, from the case $\Omega = \R^d$ applied to $\mathcal{E}^2_+ u$ and the boundedness of the extension operators we find that
\begin{align*}
[u]_{W^{1,p}(\R^d_+,w;X)} &\leq [\mathcal{E}^n_+ u]_{W^{1,p}(\R^d,w;X)} \lesssim_{p,d,[w]_{A_p}} \|\mathcal{E}^n_+ u\|_{L^p(\R^d,w;X)}^{1/2} [\mathcal{E}^n_+ u]_{W^{2,p}(\R^d,w;X)}^{1/2}
\end{align*}
Clearly, $\|\mathcal{E}^n_+ u\|_{L^p(\R^d,w;X)}\leq \|u\|_{L^p(\R^d_+,w;X)}$. Moreover, since $\partial^{\alpha}  \mathcal{E}^2_+ = E_0\partial^{\alpha}$,
\[[\mathcal{E}^n_+ u]_{W^{2,p}(\R^d,w;X)}
= \sum_{|\alpha|=2} \|E_0 \partial^{\alpha} u\|_{L^p(\R^d,w;X)}
\leq [u]_{W^{2,p}(\R^d_+,w;X)}. \]
Therefore, the result follows if we combine the two estimates.
\end{proof}

\section{Application to interpolation theory and the first derivative\label{sec:intbdr}}

For $p\in (1, \infty)$, $s\in \R$ and a weight $w\in A_p$, let $H^{s,p}_0(\R,w;X)$ denote the closure of $C^\infty_c(\R\setminus\{0\};X)$ in $H^{s,p}_0(\R,w;X)$.
In this section we characterize the complex interpolation space $[L^p(\R_+,w_{\gamma};X),H^{1,p}_0(\R_+,w_{\gamma};X)]_{\theta}$. Moreover, we use this to characterize the domains of fractional powers of the first derivative.

\subsection{Results on the whole real line}

For $k \in \N_{0}$ let
\[
W^{k+1,1}_{\mathrm{loc},0}(\R;X) := \{ f \in W^{k+1,1}_{\mathrm{loc}}(\R;X): f(0)=\ldots=f^{(k)}(0)=0 \}.
\]
Since $f(y) - f(x) = \int_{x}^{y} f'(t) \ud t$, it follows that $f$ has a version which is uniformly continuous on bounded intervals, and hence $f^{(j)}(0)$ for $j\in \{0,\ldots, k\}$ is defined in a pointwise sense

We will need the following simple lemma.
\begin{lemma}\label{lemma:indicator_pointwise_vanishing_trace}
Let $X$ be a Banach space and $k \in \N_{0}$. If $f \in W^{k+1,1}_{\mathrm{loc}}(\R;X)$ satisfies $f(0)=\ldots=f^{(k)}(0)=0$, then $1_{\R_{+}}f \in W^{k+1,1}_{\mathrm{loc}}(\R;X)$ with
\[
(1_{\R_{+}}f)^{(j)} = 1_{\R_{+}}f^{(j)}, \quad\quad j \in \{1,\ldots,k+1\}.
\]
\end{lemma}
\begin{proof}
Using an inductive argument we may reduce to the case $k=0$.
So suppose $f \in W^{1,1}_{\mathrm{loc}}(\R;X)$ satisfies $f(0)=0$.
Then $f(x) = \int_{0}^{x}f'(t) \ud t$ for all $x \in \R$, from which it follows that
\[
\one_{\R_{+}}f(x) = \int_{0}^{x}\one_{\R_{+}}f'(t) \ud t, \quad\quad x \in \R.
\]
This shows $\one_{\R_{+}}f \in W^{1,1}_{\mathrm{loc}}(\R;X)$ with $(1_{\R_{+}}f)' = 1_{\R_{+}}f'$.
\end{proof}

\begin{proposition}\label{prop:pointwisemultiplier0}
Let $X$ be a UMD Banach space, $p\in (1,\infty)$ and $\gamma\in (-1,p-1)$. Assume $s>\frac{1+\gamma}{p} -1$ and $k\in \N_0$ are such that $\frac{1+\gamma}{p} -1 + k  <s<\frac{1+\gamma}{p} +k$.
For all $f \in H^{s,p}(\R,w_\gamma;X) \cap W^{k+1,1}_{\mathrm{loc},0}(\R;X)$ we then have
\[
\| \one_{\R_+}f \|_{H^{s,p}(\R,w_\gamma;X)} \lesssim_{s,p,\gamma,X} \|f\|_{H^{s,p}(\R,w_\gamma;X)}.
\]
As a consequence, $\one_{\R_+}$ is a pointwise multiplier on $H^{s,p}_0(\R,w_\gamma;X)$.
Moreover, for all $f\in H^{s,p}_0(\R,w_\gamma;X)$ it holds that
\begin{equation}\label{eq:identityderivativeone}
(\one_{\R_+} f)^{(j)} = \one_{\R_+} f^{(j)}, \quad\quad j\in\{0, \ldots k\}.
\end{equation}
\end{proposition}

\begin{proof}
As in \cite[Proposition 3.4]{MeyVerpoint} one checks the following equivalence of extended norms on $\mathcal{S}'(\R;X)$:
\begin{equation}\label{eq:normequiv}
\begin{array}{ll}
\|f\|_{H^{s,p}(\R,w_{\gamma};X)}
&\eqsim_{s,\gamma,p,X}\: \|f\|_{H^{s-k,p}(\R,w_{\gamma};X)} + \|\partial^k f\|_{H^{s-k,p}(\R,w_{\gamma};X)} \\
&\eqsim_{s,\gamma,p,X}\: \sum_{j=0}^k \|\partial^j f\|_{H^{s-k,p}(\R,w_{\gamma};X)}.
\end{array}
\end{equation}

Let $f \in H^{s,p}(\R,w_\gamma;X) \cap W^{k+1,1}_{\mathrm{loc},0}(\R;X)$
Using \eqref{eq:normequiv}, Lemma~\ref{lemma:indicator_pointwise_vanishing_trace} and Theorem \ref{thm:pointwise_multiplier} we find
\begin{align*}
\|\one_{\R_+} f\|_{H^{s,p}(\R,w_{\gamma};X)} & \lesssim_{s,p,\gamma,X}  \|\one_{\R_+} f\|_{H^{s-k,p}(\R,w_{\gamma};X)} + \|\partial^k (\one_{\R_+} f)\|_{H^{s-k,p}(\R,w_{\gamma};X)}
\\ & =  \|\one_{\R_+} f\|_{H^{s-k,p}(\R,w_{\gamma};X)} + \|\one_{\R_+} \partial^k f\|_{H^{s-k,p}(\R,w_{\gamma};X)}
\\ & \lesssim_{s,p,\gamma,X}  \|f\|_{H^{s-k,p}(\R,w_{\gamma};X)} + \|\partial^k f\|_{H^{s-k,p}(\R,w_{\gamma};X)}
\\ & \lesssim_{s,p,\gamma,X} \|f\|_{H^{s,p}(\R,w_{\gamma};X)}.
\end{align*}
By a density argument we find that $\one_{\R_+}$ is a pointwise multiplier on $H^{s,p}_0(\R,w_\gamma;X)$.

Finally, to check that \eqref{eq:identityderivativeone} holds for $f \in H^{s,p}_0(\R,w_\gamma;X)$, observe that for $0\leq j\leq k$, by \eqref{eq:normequiv} and the above estimate
\[\|\partial^j (\one_{\R_+} f)\|_{H^{s-k,p}(\R,w_{\gamma};X)} \leq C \|\one_{\R_+}f\|_{H^{s,p}(\R,w_{\gamma};X)} \leq C \|f\|_{H^{s,p}(\R,w_{\gamma};X)}.\]
Therefore, if $f\in H^{s,p}_0(\R,w_{\gamma};X)$, then letting $f_n\in C^\infty_c(\R\setminus\{0\};X)$ be such that $f_n\to f$ in $H^{s,p}_0(\R,w_{\gamma};X)$, we find that $\partial^j (\one_{\R_+} f_n) \to \partial^j (\one_{\R_+} f)$ in $H^{s-k,p}(\R,w_{\gamma};X)$. Since $\partial^j f_n \to \partial^{j} f$ in $H^{s-k,p}(\R,w_{\gamma};X)$, by Theorem \ref{thm:pointwise_multiplier} also $\one_{\R_+}\partial^j f_n \to \one_{\R_+} \partial^{j} f$ in $H^{s-k,p}(\R,w_{\gamma};X)$.
The validity of \eqref{eq:identityderivativeone} for functions from $C^{\infty}_{c}(\R \setminus \{0\})$ and uniqueness of limits in $H^{s-k,p}(\R,w_{\gamma};X)$ yields \eqref{eq:identityderivativeone} for general $f \in H^{s,p}_0(\R,w_\gamma;X)$.
\end{proof}

\begin{proposition}\label{prop:trace}
Let $\gamma\in (-1, p-1)$ and $s \in \R$. Assume $k\in \N_0$ satisfies $k+\frac{1+\gamma}{p}<s$.
Then the following assertions hold:
\begin{enumerate}[$(1)$]
\item $\Tr_k:H^{s,p}(\R,w_{\gamma};X)\cap C^k(\R;X) \to X^{k}$ given by $\Tr_k f = (f(0), f'(0), \ldots, f^{(k)}(0))$ uniquely extends to a bounded linear mapping $\Tr_k:H^{s,p}(\R,w_{\gamma};X)\to X^{k+1}$.
\item If $f\in H^{s,p}(\R,w_{\gamma};X)$ satisfies $f|_{(0,\delta)} = 0$ or $f|_{(-\delta,0)} = 0$ for some $\delta > 0$, then $\Tr_k f = 0$.
\item There exists a bounded mapping $\ext_k:X^{k+1}\to H^{s,p}(\R,w_{\gamma};X)$ such that $\Tr_k(\ext_k)$ is the identity on $X^{k+1}$.
\end{enumerate}
\end{proposition}

\begin{proof}
We first prove (1).
By Lemma \ref{lem:densityCcH}, it is enough to establish boundedness of
\[
\Tr_k:(H^{s,p}(\R,w_{\gamma};X)\cap C^k(\R;X),\|\,\cdot\,\|_{H^{s,p}(\R,w_{\gamma};X)}) \to X^{k+1}.
\]
Choosing $x_{j}^{*} \in X^{*}$ with $\|x_{j}^{*}\| = 1$ and $\|f^{(j)}(0)\| = \langle f^{(j)}(0),x_{j}^{*}\rangle$ for each $j \in \{0,\ldots,k\}$ we have
$\langle f,x_{j}^{*}\rangle \in H^{s,p}(\R,w_{\gamma})\cap C^k(\R)$ with
\[
\|f^{(j)}(0)\| = |\langle f^{(j)}(0),x_{j}^{*}\rangle| = |\langle f,x_{j}^{*}\rangle^{(j)} (0)|, \quad \| \langle f,x_{j}^{*}\rangle \|_{H^{s,p}(\R,w_{\gamma})} \leq \|f\|_{H^{s,p}(\R,w_{\gamma};X)}.
\]
So we may restrict ourselves to the case $X=\C$.
Recall from \cite[Proposition 3.4]{MeyVerpoint} that $d/dt$ is a bounded linear operator from $H^{\sigma,p}(\R,w_{\gamma})$ to $H^{\sigma-1,p}(\R,w_{\gamma})$ for every $\sigma \in \R$.
By differentiation it thus suffices to prove that, given $\theta \in (\frac{1+\gamma}{p},\frac{1+\gamma}{p}+1)$, the following estimate holds
\[
|f(0)| \lesssim_{\theta, \gamma,p} \|f\|_{H^{\theta,p}(\R,w_{\gamma})}, \quad\quad f\in H^{\theta,p}(\R,w_{\gamma})\cap C(\R).
\]
Here we actually only need to consider $f\in H^{\theta,p}(\R,w_{\gamma})\cap C_{c}(\R)$; indeed, given $\eta \in C^{\infty}_{c}(\R)$ with $\eta(0)=1$, $f \mapsto \eta f$ defines by complex interpolation (see Proposition \ref{prop:extensionH}) a bounded linear operator on $H^{\theta,p}(\R,w_{\gamma})$ and we may consider $\eta f$ instead of $f$.
Using Lemma~\ref{lem:convolution;Bessel_pot_space} together with \cite[Theorem~1.2.19]{Grafakos1} one can check that $C^{\infty}_{c}(\R)$ is dense in $H^{\theta,p}(\R,w_{\gamma})\cap C_{c}(\R)$, where $C_{c}(\R)$ has been equipped with the supremum norm.
It thus is enough to estimate
\[
|f(0)| \lesssim_{\theta, \gamma,p} \|f\|_{H^{\theta,p}(\R,w_{\gamma};X)}, \quad\quad f\in C^{\infty}_{c}(\R).
\]
To this end, let $f \in C^{\infty}_{c}(\R) \subset \mathcal{S}(\R)$ and put $g := (1-\Delta)^{\theta/2} f\in \mathcal{S}(\R)$. Then, letting $G_{\theta} \in L^{1}(\R)$ be the kernel Lemma \ref{lem:propGs}, we find
\[
f(0) = (1-\Delta)^{-\theta/2}g\,(0) = G_{\theta}*g\,(0) = \int_\R G_\theta(x)g(-x)\ud x.
\]
By Lemma \ref{lem:propGs} we find
\[
|f(0)| \leq \int_\R |G_\theta(x)|\, |g(-x)|\ud x  \leq \|G_{\theta}\|_{L^{p'}(\R,w_{\gamma}')} \|g\|_{L^p(\R,w_{\gamma})}
\lesssim_{\theta, \gamma,p} \|f\|_{H^{\theta,p}(\R,w_{\gamma})}.
\]

To prove (2) consider the case that $f = 0$ on $(0,\delta)$. Let $\phi\in C^\infty(\R)$ be such that $\int\phi(x) \ud x = 1$ and $\phi$ is supported on $(-2, -1)$ and put $\phi_n(x) := n \phi(n x)$.
By Lemma~\ref{lem:convolution;Bessel_pot_space}, $\|\phi_n * f\|_{H^{s,p}(\R,w_{\gamma};X)} \lesssim_{p,\gamma} \|f\|_{H^{s,p}(\R,w_{\gamma};X)}$ with $\phi_n * f\to f$ in $H^{s,p}(\R,w_{\gamma};X)$.
Clearly, $\phi_n*f\in C^\infty(\R;X)$ and by the support conditions one sees that $\phi_n*f(0) = 0$ for all $n>2\delta^{-1}$. Therefore, $\Tr_k (\phi_n * f) = 0$ and the result follows by letting $n\to \infty$ and using the continuity of $\Tr_k$.

To prove (3) choose $\phi_0, \ldots, \phi_k\in C^\infty_c(\R)$ such that $\phi^{(n)}_j(0) = \delta_{jn}$ for all $0\leq j\leq k$ and $0\leq n \leq k$ and let $\ext_k (x_j)_{j=1}^k = \sum_{j=0}^k \phi_j x_j$. This clearly satisfies the required properties.
\end{proof}

We can now give a characterization of $H^{s,p}_0(\R,w_{\gamma};X)$ in terms of traces.
For it will be convenient to say that the statement $\Tr_k f = 0$ for $k \leq -1$ is empty.

\begin{proposition}\label{prop:characvanishingtrace}
Let $X$ be a Banach space, $p\in (1, \infty)$ and $\gamma\in (-1, p-1)$. Let $s \in \R$ be such that $k+\frac{1+\gamma}{p}<s<k+1 + \frac{1+\gamma}{p}$ with  $k\in \Z, k \geq -1$. Then
\[H^{s,p}_0(\R,w_{\gamma};X) = \{f\in H^{s,p}(\R,w_{\gamma};X): \Tr_k f = 0\}.\]
\end{proposition}

Note that $\Tr_k f$ is well defined  by Proposition \ref{prop:trace}.

\begin{proof}
Clearly, $\Tr_k f= 0$ for every $f\in C^\infty_c(\R\setminus\{0\};X)$. By continuity this extends to every $f\in H^{s,p}_0(\R,w_{\gamma};X)$ (see Proposition \ref{prop:trace}) and hence ``$\subseteq$'' follows. To prove the converse, let $f\in  H^{s,p}(\R,w_{\gamma};X)$ be such that $\Tr_k f = 0$.
By Lemma \ref{lem:densityCcH} we can find $\{g_{n}\}_{n \in \N} \subset C^\infty_c(\R) \otimes X$ such that
$g_{n} \to f$ in $H^{s,p}(\R,w_{\gamma};X)$ as $n \to \infty$.
Let $\ext_k$ be as constructed in the proof of Proposition \ref{prop:trace} and put $h_{n} := g_{n} - \ext_k (g_{n}^{(j)}(0))_{j=0}^k$ for each $n \in \N$.
Then $h_{n} \in \{ h \in C^\infty_c(\R) : \Tr_k h=0 \} \otimes X$ and, by Proposition \ref{prop:trace}, $h_{n} \to f-\ext_k(0)_{j=0}^{k} = f$ in $H^{s,p}(\R,w_{\gamma};X)$ as $n \to \infty$.

It remains to show that we can approximate a function $h \in C^\infty_c(\R)$ satisfying $\Tr_k h=0$
by a function in $C^\infty_c(\R\setminus\{0\})$ with respect to the norm of $H^{s,p}(\R,w_{\gamma})$. Writing $h = \one_{\R_+} h + \one_{\R_-} h=:h_0 + h_1$, it follows from Proposition \ref{prop:pointwisemultiplier0} that $h_0, h_1\in H^{s,p}(\R,w_{\gamma};X)$ and hence it suffices to approximate each of the terms $h_0$ and $h_1$. Fix $\phi\in C_c^\infty(\R)$ with $\int_{\R} \phi\,dx= 1$ and $\supp \phi \subseteq [1, \infty)$ and define $\phi_n := n\phi(n\cdot)$ for each $n\in \N$. Then $\phi_n * h_0 \in C^\infty_c(\R\setminus\{0\})$ with $\phi_n * h_0 \to h_0$ in $H^{s,p}(\R,w_{\gamma})$ as $n \to \infty$ by Lemma \ref{lem:convolution;Bessel_pot_space}.
A similar argument can be used for $h_1$.
\end{proof}

We can now prove the main result of this section:
\begin{theorem}\label{thm:interpR0}
Let $X$ be a UMD space and $\gamma\in (-1,p-1)$. Let $\theta\in (0,1)$ and $s_0, s_1>-1+\frac{\gamma+1}{p}$. Let $s = s_0 (1-\theta) + s_1 \theta$.
If $s_0, s_1, s\notin \N_0 + \frac{\gamma+1}{p}$, then
\begin{equation}\label{eq:interpH0}
[H^{s_0,p}_0(\R,w_{\gamma};X), H^{s_1,p}_0(\R,w_{\gamma};X)]_{\theta} = H^{s,p}_0(\R,w_{\gamma};X).
\end{equation}
\end{theorem}
\begin{proof}
Assume $s_0, s_1, s\notin \N_0 + \frac{\gamma+1}{p}$ and let $E^{\sigma,p}_{\rm prod} := H^{\sigma,p}_{\R_+}(\R,w_{\gamma};X) \times H^{\sigma,p}_{\R_-}(\R,w_{\gamma};X)$, $\sigma \in \R$, for shorthand notation.

Let $\sigma > -1+\frac{\gamma+1}{p}$ with $\sigma \notin \N_0 + \frac{\gamma+1}{p}$.
By Proposition \ref{prop:trace} $\Tr_k$ vanishes on $H^{\sigma,p}_{\R_\pm}(\R,w_{\gamma};X)$ for integers $k\in [0,\sigma-\tfrac{\gamma+1}{p})$.
Thus, in view of Proposition~\ref{prop:characvanishingtrace}, the map
\[R: E^{\sigma,p}_{\rm prod} \to H^{\sigma,p}_0(\R,w_{\gamma};X), \qquad R(g,h) := g+h,\]
is a well-defined contraction.
That the map
$$S: H^{\sigma,p}_0(\R,w_{\gamma};X) \to E^{\sigma,p}_{\rm prod} ,\qquad Sf :=(\one_{\R_+^d} f, \one_{\R_-^d} f),$$
is well-defined and continuous follows from Propositions \ref{prop:pointwisemultiplier0} and \ref{prop:characvanishingtrace}.
Since $R^{-1} = S$, the result follows from Proposition \ref{prop:interpH}.
\end{proof}

\subsection{Results on the positive half line}

Let $\gamma\in (-1, p-1)$ and $s \in \R$. Assume $k\in \N_0$ satisfies $k+\frac{1+\gamma}{p}<s$.
By Proposition~\ref{prop:trace}, if $\tilde{f}_{1},\tilde{f}_{2} \in H^{s,p}(\R,w_{\gamma};X)$ satisfy $\tilde{f}_{1|\R_{+}} = \tilde{f}_{2|\R_{+}}$, then $\Tr_{k}\tilde{f}_{1}=\Tr_{k}\tilde{f}_{2}$.
Therefore, $\Tr_{k}: H^{s,p}(\R,w_{\gamma};X)\to X^{k+1}$ gives rise to a well-defined bounded linear operator $\Tr_{k,+}:H^{s,p}(\R_+,w_{\gamma};X)\to X^{k+1}$ given by $\Tr_{k,+}f=\Tr_{k}\tilde{f}$ whenever $\tilde{f}_{|\R_{+}} = f$.
After reducing to the scalar-valued case, Proposition~\ref{prop:extensionH} shows that
\begin{equation}\label{eq:tracekfpoint}
\Tr_{k,+}f = (f(0), f'(0), \ldots, f^{(k)}(0)), \quad\quad f\in H^{s,p}(\R_+,w_{\gamma};X)\cap C^{k}([0,\infty);X);
\end{equation}
in the case $X=\C$ we simply pick the least integer $m \geq |s|$ and observe that $\Tr_{k,+} = \Tr_{k} \circ \mathcal{E}_+^m $.

Let $H^{s,p}_0(\R_+,w_{\gamma};X)$ denote the closure of $C^\infty_c((0,\infty);X)$ in $H^{s,p}(\R_+,w_{\gamma};X)$.

\begin{proposition}\label{prop:characvanishingtrace2}
Let $X$ be a Banach space, $p \in (1,\infty)$, $\gamma\in (-1,p-1)$ and $s \in \R$.
Assume $k \in \N_{0}$ satisfies $k+\frac{1+\gamma}{p}<s<k+1 + \frac{1+\gamma}{p}$.
Then
\[H^{s,p}_0(\R_+,w_{\gamma};X) = \{f\in H^{s,p}(\R_+,w_{\gamma};X): \Tr_{k,+} f = 0\}.\]
\end{proposition}
\begin{proof}
Clearly, $\subseteq$ holds.
To prove the converse let  $f\in H^{s,p}(\R_+,w_{\gamma};X)$ be such that $\Tr_{k,+} f = 0$.
Pick $\tilde{f} \in H^{s,p}(\R,w_{\gamma};X)$ with $\tilde{f}_{|\R_{+}} = f$.
Then $\Tr_{k}\tilde{f} = \Tr_{k,+} f = 0$. By Proposition~\ref{prop:characvanishingtrace} we thus get $\tilde{f} = \lim_{n \to \infty}\tilde{f}_{n}$ in $H^{s,p}(\R,w_{\gamma};X)$ for some sequence $(\tilde{f}_{n})_{n \in \N}$ from $C^{\infty}_{c}(\R \setminus \{0\};X)$. Now $f_{n}:= \tilde{f}_{n|\R_{+}} \in C^\infty_c((0,\infty);X)$ with $f_{n} \to f$ in $H^{s,p}(\R_+,w_{\gamma};X)$ as $n \to \infty$.
\end{proof}

\begin{theorem}\label{thm:compl_int_bd-cond_half-space}
Let $X$ be a UMD space, $p \in (1,\infty)$ and $\gamma\in (-1,p-1)$. Let $\theta\in (0,1)$ and $s_0, s_1>-1+\frac{\gamma+1}{p}$. Let $s = s_0 (1-\theta) + s_1 \theta$.
If $s_0, s_1, s\notin \N_0 + \frac{\gamma+1}{p}$, then
\begin{equation}\label{eq:interpH0_half-line}
[H^{s_0,p}_0(\R_+,w_{\gamma};X), H^{s_1,p}_0(\R_+,w_{\gamma};X)]_{\theta} = H^{s,p}_0(\R_+,w_{\gamma};X).
\end{equation}
\end{theorem}
\begin{proof}
Let $m$ be the least integer such that $m\geq \max\{|s_0|, |s_1|\}$.
For each $\sigma > -1+\frac{\gamma+1}{p}$ with $|\sigma| \leq m$  and $\sigma \notin \N_0 + \frac{\gamma+1}{p}$,
\[
S:H^{\sigma,p}_0(\R_+,w_{\gamma};X)\to  H^{\sigma,p}_0(\R,w_{\gamma};X), \quad Sf:= \mathcal{E}_+^m f,
\]
is a well-defined bounded linear operator thanks to Propositions \ref{prop:characvanishingtrace} and \ref{prop:characvanishingtrace2}.
For each $\sigma \in \R$, let $R:H^{\sigma,p}_0(\R,w_{\gamma};X)\to H^{\sigma,p}_0(\R_+,w_{\gamma};X)$ denote the restriction operator.
Using Theorem~\ref{thm:interpR0}, the proof can now be completed as in Proposition~\ref{prop:interpH}~(2).
\end{proof}

\subsection{Fractional domain spaces}

For $p\in (1, \infty)$ and $\gamma\in (-1,p-1)$ let
\[W^{k,p}_0(\R_+,w_{\gamma};X) = \{f\in W^{k,p}(\R_+,w_{\gamma};X): f(0) = f^{(1)}(0)=\ldots = f^{(k-1)}(0) = 0\}.\]
If $X$ is a UMD space, then it follows from Propositions \ref{prop:W=H2}, \ref{prop:characvanishingtrace2} and \eqref{eq:tracekfpoint} that \begin{equation}\label{eq:W0H0}
W^{k,p}_0(\R_+,w_{\gamma};X) = H^{k,p}_0(\R_+,w_{\gamma};X).
\end{equation}

Let us now briefly recall the $H^{\infty}$-calculus for sectorial operators, for which there are several conventions in the literature.
For a survey and an extensive treatment of the subject we refer the reader to \cite{Weis2006_survey} and \cite{Haase:2,HNVW2,KuWe}, respectively.

For each $\theta \in (0,\pi)$ we define the sector
\[
\Sigma_{\theta} := \{ \lambda \in \C \setminus \{0\} : |\arg(\lambda)| < \theta \}.
\]
A closed densely defined linear operator $(A, D(A))$ on $X$
is said to be {\em sectorial of type} $\sigma \in (0,\pi)$ if it is injective and has dense range, $\Sigma_{\pi-\sigma} \subset \rho(-A)$, and for all $\sigma' \in (\sigma,\pi)$
\[
\sup\{ \| \lambda(\lambda+A)^{-1}\| : \lambda \in  \Sigma_{\pi-\sigma'}\} < \infty.
\]
The infimum of all $\sigma \in (0,\pi)$ such that $A$ is sectorial of type $\sigma$ is called the {\em sectoriality angle} of $A$ and is denoted by $\phi_{A}.$

Let $H^{\infty}(\Sigma_{\theta})$ denote the Banach space of all bounded analytic functions $f : \Sigma_{\theta} \to \C$, endowed with the supremum norm. Let $H^{\infty}_{0}(\Sigma_{\theta})$ denote its linear subspace of all $f$ for which there exists $\epsilon > 0$ and $C \geq 0$ such that
\[
|f(z)| \leq \frac{C|z|^{\varepsilon}}{(1 + |z|)^{2\varepsilon}} , z \in \Sigma_{\theta}.
\]
If $A$ is sectorial of type $\sigma_{0} \in (0,\pi)$, then for all $\sigma \in (\sigma_{0},\pi)$
and $f \in H^{\infty}_{0}(\Sigma_{\sigma})$ we define the bounded linear operator $f(A)$ by
\[
f(A) := \frac{1}{2\pi\imath} \int_{\partial\Sigma_{\sigma}}f(z)(z+A)^{-1} \ud z.
\]

A sectorial operator $A$ of type $\sigma_{0} \in (0,\pi)$ is said to have a {\em bounded $H^{\infty}(\Sigma_{\sigma})$-calculus} for $\sigma \in (\sigma_{0},\pi)$ if there exists a $C \in [0,\infty)$ such that
\[
\| f(A) \| \leq \|f\|_{H^{\infty}(\Sigma_{\sigma})}, \qquad f \in H^{\infty}_{0}(\Sigma_{\sigma}).
\]
In this case the mapping $f \mapsto f(A)$ extends to a bounded algebra homomorphism from $H^{\infty}(\Sigma_{\sigma})$ to $\mathcal{B}(X)$ of norm $\leq C$.
The $H^{\infty}$-angle of $A$ is defined as the infimum of all $\sigma$ for which $A$ has a bounded $H^{\infty}(\Sigma_{\sigma})$-calculus and is denoted by $\phi^{\infty}_{A}$.

Below we will make use of the following fact.
Let $A$ be an operator on a reflexive Banach space $X$. If $A$ is a sectorial operator having a bounded $H^{\infty}$-calculus, then so is $A^{*}$ with $\phi^{\infty}_{A}=\phi^{\infty}_{A^{*}}$.

\begin{theorem}\label{thm:fractional_domains}
Let $X$ be a UMD space, $p \in (1,\infty)$ and $\gamma\in (-1,p-1)$.
\begin{enumerate}
\item\label{item:thm:fractional_domains;bd_cond} The realization of $\partial_{t}$ on $L^{p}(\R_{+},w_{\gamma};X)$ with domain $W^{1,p}_{0}(\R_{+},w_{\gamma};X)$ has a bounded $H^\infty$-calculus of angle $\pi/2$ with $\Do(\partial_{t}^{s}) = H^{s,p}_{0}(\R_+,w_{\gamma};X)$ for every $s > 0$ with $s \notin \frac{1+\gamma}{p}+\N_{0}$.
\item\label{item:thm:fractional_domains;no_bd_cond} The realization of $-\partial_{t}$ on $L^{p}(\R_{+},w_{\gamma};X)$ with domain $W^{1,p}(\R_{+},w_{\gamma};X)$ has a bounded $H^\infty$-calculus of angle $\pi/2$ with $\Do((-\partial_{t})^{s}) = H^{s,p}(\R_+,w_{\gamma};X)$ for every $s > 0$.
\end{enumerate}
\end{theorem}

For $\gamma\in [0,p-1)$ the case $\frac{d}{dt}$ follows from \cite[Theorem 4.5]{PrSi}. For $\gamma\in [0,p-1)$ the case $-\frac{d}{dt}$ follows from \cite[Theorem 2.7]{MeySchn}. Below we present a proof that works for all $\gamma\in (-1, p-1)$, in which \eqref{item:thm:fractional_domains;bd_cond} is derived from \eqref{item:thm:fractional_domains;no_bd_cond} by a simple duality argument.

\begin{proof}
Let us first establish the assertions regarding the $H^{\infty}$-calculus.
We start with \eqref{item:thm:fractional_domains;no_bd_cond}, from which we will derive \eqref{item:thm:fractional_domains;bd_cond} by duality.

For \eqref{item:thm:fractional_domains;no_bd_cond} we denote by $A$ the realization of $-\partial_{t}$ on $L^{p}(\R_{+},w_{\gamma};X)$ with domain $W^{1,p}(\R_{+},w_{\gamma};X)$ and by $\tilde{A}$ the realization of $-\partial_{t}$ on $L^{p}(\R,w_{\gamma};X)$ with domain $W^{1,p}(\R,w_{\gamma};X)$. As in \cite[Example 10.2]{KuWe}, using Proposition \ref{prop:Mihlin}, one can show that $\tilde{A}$ has a bounded $H^\infty$-calculus of angle $\pi/2$.
So it is enough to show that $\C_{+} \subset \rho(-A)$ with
\[
(\lambda+A)^{-1}f = R(\lambda+\tilde{A})^{-1}Ef =: S(\lambda)f, \qquad \lambda \in \C_{+}, f \in L^{p}(\R_{+},w_{\gamma};X),
\]
where $E  \in \mathcal{B}(L^{p}(\R_{+},w_{\gamma};X),L^{p}(\R,w_{\gamma};X))$ is the extension by zero operator, and $R$ denotes the operator of restriction from $\R$ to $\R_{+}$.
For each $\lambda \in \C_{+}$, $S(\lambda)$ defines a linear operator from $L^{p}(\R_{+},w_{\gamma};X)$ to $W^{1,p}(\R_{+},w_{\gamma};X)$ with the property that $(\lambda+A)S(\lambda)=I$. So, fixing $\lambda \in \C_{+}$, we only need to show that $\ker(\lambda+A)=\{0\}$. To this end, let $u \in W^{1,p}(\R_{+},w_{\gamma};X)$ satisfy $(\lambda-\partial_{t})u=0$.
By basic distribution theory (cf.\ \cite[Theorem~9.4]{Duistermaat&Kolk_distributies}) we find that $u$ is a classical solution in the sense that $u \in C^{\infty}(\R_{+};X)$ with $u'=\lambda u$, implying that $u = c\exp(\lambda\,\cdot\,)$ for some $c \in X$.
Since $\exp(\lambda\,\cdot\,) \notin L^{p}(\R_{+},w_{\gamma})$, it follows that $u=0$.

For \eqref{item:thm:fractional_domains;bd_cond} we denote by $A$ the realization of $\partial_{t}$ on $L^{p}(\R_{+},w_{\gamma};X)$ with domain $W^{1,p}_{0}(\R_{+},w_{\gamma};X)$ and by $B$ the realization of $-\partial_{t}$ on $L^{p'}(\R_{+},w_{\gamma'};X^{*})$ with domain $W^{1,p'}(\R_{+},w_{\gamma'};X^{*})$.
Recall that $[L^{p}(\R_{+},w_{\gamma};X)]^{*} = L^{p'}(\R_{+},w_{\gamma'};X^{*})$ with respect to the natural pairing (see  \cite[Proposition 3.5]{MeyVerpoint}), $X$ being reflexive as a UMD space (see \cite[Theorem 4.3.3]{HNVW1}).
Integration by parts (see Lemma~\ref{lem:integration_by_parts} below) yields $A \subset B^{*}$.
By \eqref{item:thm:fractional_domains;no_bd_cond} (and the fact that duals of UMD spaces are again UMD) it is enough to establish the reverse.
By \cite[Exercise 1.21(4)]{EN}, for the latter it suffices that $\lambda+A$ is surjective and $\lambda+B^{*}$ is injective for some $\lambda \in \C$. To this end, let us establish this for some fixed $\lambda \in \C_{+}$. Then $\lambda \in \rho(-B)=\rho(-B^{*})$ by \eqref{item:thm:fractional_domains;no_bd_cond}; in particular, $\lambda+B^{*}$ is injective.
As in \eqref{item:thm:fractional_domains;no_bd_cond} we can find a linear operator $S(\lambda): L^{p}(\R_{+},w_{\gamma};X) \to W^{1,p}(\R_{+},w_{\gamma};X)$ such that
$(\lambda+A)S(\lambda) = I$. Then the operator $T(\lambda) : L^{p}(\R_{+},w_{\gamma};X) \to W^{1,p}_{0}(\R_{+},w_{\gamma};X)$ given by
\[
T(\lambda)f:= S(\lambda)f - [S(\lambda)f](0)\exp(-\lambda\,\cdot\,),
\]
satisfies $(\lambda+A)T(\lambda)= I$, which shows that $\lambda+A$ is surjective.

Finally we will identify the fractional domain spaces. From the definitions it follows that $\Do(\partial_{t}^{k}) = W^{k,p}_{0}(\R_+,w_{\gamma};X)$ and $\Do((-\partial_{t})^{k}) = W^{k,p}(\R_+,w_{\gamma};X)$ as sets for every $k \in \N$. Moreover, it follows from Lemma \ref{lem:normequivalenceSobolevhighestorder} and Young's inequality for products that there is also an equivalence of norms.
The assertions concerning the fractional domain spaces subsequently follow from \cite[Theorem 6.6.9]{Haase:2}, Proposition~\ref{prop:W=H2} and Theorem~\ref{thm:compl_int_bd-cond_half-space}.
\end{proof}

\begin{lemma}[Integration by parts]\label{lem:integration_by_parts}
Let $X$ be a Banach space, $p \in (1,\infty)$ and $w \in A_{p}$. For all $u \in W^{1,p}(\R_{+},w;X)$ and $v \in W^{1,p'}(\R_{+},w';X^{*})$, where $w'=w^{-\frac{1}{p-1}}$ is the $p$-dual weight of $w$, there holds the integration by parts identity
\[
\langle u',v \rangle_{\langle L^{p}(\R_{+},w;X),L^{p'}(\R_{+},w';X) \rangle} =  -u(0)v(0) - \langle u,v' \rangle_{\langle L^{p}(\R_{+},w;X),L^{p'}(\R_{+},w';X) \rangle}.
\]
\end{lemma}
\begin{proof}
By the remark preceding this lemma and Lemma~\ref{lem:densityCcW}, $C^{\infty}_{c}(\overline{\R}_{+}) \otimes X$ is dense in $W^{1,p}(\R_{+},w;X)$ and $C^{\infty}_{c}(\overline{\R}_{+}) \otimes X^{*}$ is dense in $W^{1,p'}(\R_{+},w';X^{*})$. The desired result thus follows from integration by parts for functions from $C^{\infty}_{c}(\overline{\R}_{+})$.
\end{proof}

\def\cprime{$'$}

\end{document}